\documentclass[11pt]{amsart}

\usepackage{amsthm} 
\usepackage{amssymb}
\usepackage{graphicx}									 
\usepackage{stmaryrd}                  
\usepackage{footmisc}                  
\usepackage{paralist}
\usepackage{wrapfig}
\usepackage[draft]{pdfcomment}
\usepackage{bbm}
\usepackage[ngerman,british]{babel}
\usepackage[arrow, matrix, curve]{xy}
\usepackage{color}

\usepackage{enumitem}

\newtheorem*{thm-plain}{Theorem}

\newtheorem{thm}{Theorem}[section]
\newtheorem*{thmw}{Theorem}
\newtheorem{lem}[thm]{Lemma}

\newtheorem{prp}[thm]{Proposition}

\theoremstyle{definition}
\newtheorem{dfn}[thm]{Definition}

\newtheorem*{thmA}{Theorem A}
\newtheorem*{thmB}{Theorem B}

\theoremstyle{remark}
\newtheorem{rem}[thm]{Remark}

\newtheorem*{conj}{Conjecture}

\newcommand{\Co}{\mathbb{C}}
\newcommand{\R}{\mathbb{R}}
\newcommand{\RP}{\mathbb{RP}}

\newcommand{\Z}{\mathbb{Z}}

\newcommand{\SOr}{\mathrm{SO}}

\newcommand{\Or}{\mathrm{O}}

\newcommand{\To}{\rightarrow}
\newcommand{\LTo}{\longrightarrow}

\newcommand{\LMTo}{\longmapsto}

\newcommand{\Gp}{G^{\scriptscriptstyle+}}
\newcommand{\Gt}{G^{\scriptscriptstyle\times}}
\newcommand{\Gs}{G^{*}}

\newcommand{\Orb}{\mathcal{O}}
	
\title[On metrics on 2-orbifolds all of whose geodesics are closed]{On metrics on 2-orbifolds\\ all of whose geodesics are closed}
\author{Christian Lange}
\address{Christian Lange, Mathematisches Institut der Universit\"at zu K\"oln, Weyertal 86-90, 50931 K\"oln, Germany}
\email{clange@math.uni-koeln.de}
\thanks{The author is partially supported by the DFG funded project SFB/TRR 191 `Symplectic Structures in Geometry, Algebra and Dynamics'.}

\begin{document}

\begin{abstract} We show that the periods and the topology of the space of closed geodesics on a Riemannian 2-orbifold all of whose geodesics are closed depend, up to scaling, only on the orbifold topology and compute it. In the manifold case we recover the fact proved by Gromoll, Grove and Pries that all prime geodesics have the same length, without referring to the existence of simple geodesics. We partly strengthen our result in terms of conjugacy of contact forms and explain how to deduce rigidity on the projective plane based on a systolic inequality due to Pu.
\end{abstract}\maketitle	

\section{Introduction}
\label{sec:Pse_Introduction}
 
Riemannian manifolds all of whose geodesics are closed have been studied since the beginning of the $20$th century, when the first nontrivial examples were constructed by Tannery and Zoll. The famous book of Besse \cite{MR0496885} still describes the state of knowledge of the subject to a large extent. Some notable exceptions are concerned with relations between the lengths of geodesics on a Riemannian manifold all of whose geodesics are closed, henceforth called Besse manifold, and its topology. For instance a conjecture of Berger, stating that on a simply connected Besse manifold all prime geodesics have the same length, was proved by Gromoll and Grove for $2$-spheres \cite{MR0636885} and recently by Radeschi and Wilking for all topological spheres of dimension at least $4$ \cite{RadWil}. Apart from spheres, which admit many Besse metrics, i.e. Riemannian metrics all of whose geodesics are closed, the only known Besse manifolds are the other compact rank one symmetric spaces. 
Moreover, it was shown by Pries that the conclusion of Berger's conjecture also holds for the real projective plane \cite{MR2481742}, i.e. that all prime geodesics of a Besse metric have the same length.
 
Only little is known in the more general setting of Riemannian orbifolds. We define a \emph{Besse metric} on an orbifold as a Riemannian orbifold metric all of whose orbifold geodesics are closed and a \emph{Besse orbifold} as an orbifold endowed with a Besse metric (cf. Section \ref{sub:Riem_orb}). On Besse orbifolds new phenomena occur that are not present in the manifold case. For instance, Berger's conjecture does not hold for so-called spindle orbifolds \cite{MR2529473}, which admit many Besse metrics (see Section \ref{sub:P-metrics_on_spindle}). However, it turns out that there is still a relation between the periods of geodesics on a Besse $2$-orbifold and its topology. In fact, we generalize the results of Gromoll, Grove and Pries mentioned above in a unifying approach to the setting of Riemannian $2$-orbifolds. We prove the following result.

\begin{thmA} The geodesic periods of a Besse $2$-orbifold are determined up to scaling by the orbifold topology. In the manifold case all prime geodesics have the same length.
\end{thmA}

The \emph{geodesic periods} of a Besse 2-orbifold can be thought of as the set of lengths of prime (orbifold) geodesics counted with multiplicity. Note, however, that we prove the theorem for a slightly more general notion of geodesic periods; see Definition \ref{dfn:period_spectrum}. By the \emph{orbifold topology} we mean the orbifold diffeomorphism type, which can, in the case of $2$-orbifolds, be encoded by finitely many numerical invariants (see Section \ref{sub:Riem_orb}).

In the manifold case the proofs by Gromoll, Grove and Pries hinge on the existence of at least three simple closed geodesics, i.e. closed geodesics without self-intersections (cf. Remark \ref{rem:Lusternik_Schnirelmann}). Using a connectedness argument they moreover show that all prime geodesics are simple. This observation combined with the Blaschke conjecture for $S^2$ proved by Green \cite{Green} shows that a Besse metric on the real projective plane has constant curvature \cite{MR2481742}. Our proof is independent of the existence of simple closed geodesics. Moreover, we show how rigidity on the real projective plane can be deduced from our result based on a systolic inequality due to Pu and the fact that the geodesic flows of any two Besse metrics on the real projective plane are conjugated by a contactomorphism (see Section \ref{sub:rigidity}). The conjugacy of geodesic flows of Besse metrics is shown to hold on any $2$-orbifold with isolated singularities (see Section \ref{sub:conjugacy_contact_structures}), generalizing the case of the $2$-sphere considered in \cite{ABHS16} .


The paper is structured as follows. After reviewing some preliminaries and examples, we first prove that \emph{a $2$-orbifold admits a Besse metric if and only if it is either bad or spherical}, in other words, if and only if its orbifold Euler characteristic is positive (Proposition \ref{prp:admit_Besse}). Moreover, in Section \ref{sub:P-metrics_on_spindle} we explain that in many cases there exists an abundance of Besse metrics. The space of oriented prime geodesics on a Besse $2$-orbifold $\Orb$ has a natural orbifold structure $\Orb_g$ and admits a natural involution $i$ coming from time reversal. We call $\Orb_g$ the \emph{orbifold of oriented geodesics} and $\Orb_g/i$ the \emph{orbifold of non-oriented geodesics}, and prove the following rigidity result.

\begin{thmB} For a Besse $2$-orbifold $\Orb$ the orbifolds of oriented and non-oriented geodesics are determined by the orbifold topology of $\Orb$. More precisely, the following cases can occur.
\begin{enumerate}
\item	$\Orb \cong S^2/G$, $\Orb_g \cong S^2/\Gt$ and $\Orb_g/i \cong S^2/\Gs$ as orbifolds where $G<\Or(3)$ is a finite subgroup, $\Gt= \{\mathrm{det}(g)g| g \in G\}<\SOr(3)$ and $\Gs= \left\langle G, -1 \right\rangle<\Or(3)$.
\item $pq$ odd, $\Orb \cong S^2(p,q)$, $\Orb_g \cong S^2((p+q)/2,(p+q)/2)$ and $\Orb_g/i \cong \RP^2((p+q)/2)$.
\item $pq$ even, $\Orb \cong S^2(p,q)$, $\Orb_g \cong S^2((p+q)/\kappa,(p+q)/\kappa)$ and $\Orb_g/i \cong D^2((p+q)/\kappa;)$ with $\kappa$ being $1$ or $2$ depending on whether $p+q$ is odd or even.
\item $pq$ odd, $\Orb \cong D^2(;p,q)$, $\Orb_g \cong S^2(2,2,(p+q)/2)$ and $\Orb_g/i \cong D^2(2;(p+q)/2)$.
\item $pq$ even, $\Orb \cong D^2(;p,q)$, $\Orb_g \cong S^2(2,2,(p+q)/\kappa)$ and $\Orb_g/i \cong D^2(;2,2,(p+q)/\kappa)$ with $\kappa$ being $1$ or $2$ depending on whether $p+q$ is odd or even.
\end{enumerate}
\end{thmB}

For explanations on the notations we refer to Section \ref{sec:pre}. The covering $\Orb_g\To \Orb_g/i$ encodes information on the geodesic periods of $\Orb$ (see Section \ref{sub:orb_of_geo}). In this way we will apply Theorem B in the proof of Theorem A.  However, in Section \ref{sub:orb_of_geo} we show by example that the geodesic periods and the orbifolds of geodesics in general do not determine each other. For the proof of Theorem A in the case of orbifolds with non-isolated singularities additional geometric arguments are required (see Section \ref{sub:non-orient_spher}). The proof of Theorem B relies on the following ideas. The unit tangent bundle $M=T^1\Orb$ of a Besse $2$-orbifold $\Orb$ with isolated singularities is a manifold and the geodesic flow on it is periodic due to a result of Epstein \cite{MR0288785}. We obtain two transversal Seifert fiberings on $M$, one from the geodesic flow and another one from the projection $T^1\Orb \To \Orb$. Properties of these Seifert fiberings and their interplay imply the claim in many cases. For Besse $2$-orbifolds with codimension one singularities the result is obtained by considering the orientable double cover. An explicit list of all $2$-orbifolds admitting a Besse metric together with their orbifolds of geodesics and their geodesic periods can be found in the appendix, Table \ref{tab:gls}.

Since our approach does not rely on the existence of simple closed geodesics, it also works in more general Hamiltonian settings \cite{FraLaSuhr}. In \cite{FraLaSuhr} a Hamiltonian version of the result of Gromoll and Grove is proven, which could not have been obtained along the lines of the original proof. However, note that in general our result is a Riemannian phenomenon that cannot be seen from the Hamiltonian point of view. For instance, the real projective plane and the teardrop $S^2(3)$ have the same unit tangent bundle (see Lemma \ref{lem:gluing} and Section \ref{sub:pro_plane}), but the geodesic periods of Besse metrics on them differ (see appendix, Table \ref{tab:gls}).
\newline
\newline
\emph{Acknowledgements.} The author would like to thank Alexander Lytchak for drawing his attention to the subject. He is grateful to Alberto Abbondandolo for explaining to him how this paper's result combined with work by Pu implies rigidity on the real projective plane (see Section \ref{sub:rigidity}). He thanks the referee for critical comments and suggestions that helped to improve the exposition.

The research in this paper is part of a project in the SFB/TRR 191 `Symplectic Structures in Geometry, Algebra and Dynamics'. A partial support is gratefully acknowledged.

\section{Preliminaries}\label{sec:pre}

\subsection{Orbifolds} \label{sub:Riem_orb}
For a definition of a (smooth) orbifold we refer to \cite{MR1744486} or \cite{MR2883685}. A Riemannian orbifold can be defined as follows.
\begin{dfn}\label{dfn:orbifold} An \emph{$n$-dimensional Riemannian orbifold} $\Orb$ is a length space such that for each point $x \in \Orb$, there exists a neighborhood $U$ of $x$ in $\Orb$, an $n$-dimensional Riemannian manifold $M$ and a finite group $\Gamma$ that acts by isometries on $M$ such that $U$ with the restricted metric and $M/\Gamma$ with the quotient metric are isometric.
\end{dfn}
A \emph{length space} is a metric space in which the distance of any two points can be realized as the infimum of the lengths of all rectifiable paths connecting these points \cite{MR1835418}. Behind the above definition lies the fact that an effective isometric action of a finite group on a simply connected Riemannian manifold can be recovered from the corresponding metric quotient. In the case of spheres this is proven in \cite{MR1935486}; the general case follows in a similar way (cf. e.g. \cite{Lange}). In particular, a Riemannian orbifold in the above sense admits a smooth orbifold structure and a compatible Riemannian structure that in turn induces the metric structure. For a point $x$ on a Riemannian orbifold, the isotropy group of a preimage of $x$ in a Riemannian manifold chart is uniquely determined up to conjugation. Its conjugacy class in $\Or(n)$ is denoted as $\Gamma_x$ and is called the \emph{local group} of $\Orb$ at $x$. Riemannian orbifolds are stratified by manifolds. The $k$-dimensional stratum consists of those points $x \in \Orb$ for which $\mathrm{dim}(\mathrm{Fix}(\Gamma_x))=k$.

The underlying topological space $|\Orb|$ of a $2$-orbifold $\Orb$ is a manifold with boundary. In this case the orbifold $\Orb$ is \emph{orientable} if and only if $|\Orb|$ is an orientable surface without boundary. A $2$-orbifold can have three types of singularities. Isolated singularities in the interior of $|\Orb|$ whose local groups are cyclic and orientation preserving, mirror singularities on the boundary of $|\Orb|$ whose local groups are generated by a reflection, and corner-reflector singularities on the boundary of $|\Orb|$ whose local groups are dihedral groups generated by two distinct reflections. The closure of the $1$-dimensional stratum is the boundary of $|\Orb|$ and consists of the mirror- and the corner-reflector singularities. We denote a $2$-orbifold $\Orb$ with $l$ isolated singularities in the interior of $|\Orb|$ (whose local group are) of order $n_1,\ldots,n_l$ and $k$ isolated corner-reflector singularities on the boundary of $|\Orb|$ whose local groups are dihedral of order $2m_1,\ldots,2m_k$ by $\Orb=|\Orb|(n_1,\ldots,n_l;m_1,\ldots,m_k)$. If the boundary of $|\Orb|$ is empty we simply write $\Orb=|\Orb|(n_1,\ldots,n_l)$. We will use the following conventions. If two Riemannian orbifolds $\Orb$ and $\Orb'$ are isometric, we write $\Orb=\Orb'$. If two (smooth) orbifolds are isomorphic, we write $\Orb\cong \Orb'$. For $2$-dimensional orbifolds $\Orb$ and $\Orb'$ we have $\Orb\cong \Orb'$, if and only if the underlying topological spaces of $\Orb$ and $\Orb'$ are homeomorphic and the numerical invariants associated to $\Orb$ and $\Orb'$ above coincide. 

The metric quotient of a Riemannian $2$-orbifold by a finite group of isometries is again a Riemannian $2$-orbifold. If $\Orb$ is a Riemannian $2$-orbifold with non-empty $1$-dimensional stratum, then its \emph{metric double} $\hat \Orb$ along the closure of this stratum (see \cite[Def.~3.1.12]{MR1835418}) is a Riemannian orbifold with isolated singularities. Moreover, in this case  the natural involution of $\hat \Orb$ is an isometry with quotient $\Orb$. 

We are interested in (orbifold) geodesics in the following sense.
\begin{dfn}\label{dfn:geodesics} A \emph{geodesic} on a Riemannian orbifold is a continuous path that can locally be lifted to a geodesic in a Riemannian manifold chart. A \emph{closed geodesic} is a continuous loop that is a geodesic on each subinterval. A \emph{prime geodesic} is a closed geodesic that is not a concatenation of nontrivial closed geodesics. 
\end{dfn}

In particular, we are interested in Riemannian orbifold metrics all of whose geodesics ``are closed'', i.e. factor through closed geodesics. We call such metrics \emph{Besse}. We call a Riemannian orbifold whose metric is Besse a \emph{Besse orbifold}.
\begin{dfn}\label{dfn:periods} By the \emph{period} of a closed geodesic we mean its length as a parametrized curve.
\end{dfn}
We stress this definition since the length of a geodesic on a Riemannian orbifold as a parametrized curve may differ from the length of its geometric image. To illustrate this and to provide some familiarity with geodesics on a Riemannian $2$-orbifold $\Orb$ let us summarize some of their properties. In the regular part geodesics behave like ordinary geodesics in a Riemannian manifold. A geodesic hitting an isolated singularity is either reflected or goes straight through it depending on whether the order of the singularity is even or odd. In particular, a closed geodesic that hits an isolated singularity of even order passes its trajectory twice during a single period. Moreover, we see that an orbifold geodesic is in general not locally length minimizing. However, on the other hand, a locally length minimizing path, which has to lie completely in either the 1- or the 2-dimensional stratum, is locally covered by length minimizing paths in Riemannian manifold charts and hence an orbifold geodesic. Suppose the $1$-dimensional stratum of $\Orb$ is non-empty and let $\hat \Orb$ be the metric double of $\Orb$ along the closure of this stratum. Then a geodesic hitting the topological boundary of $\Orb$ continues as the projection to $\Orb$ of its continuation in $\hat \Orb$. In particular, a geodesic hitting the $1$-dimensional stratum is reflected according to the usual reflection law. A geodesic that remains in the closure of the $1$-dimensional stratum can never leave it.

We need the following concept.

\begin{dfn} A \emph{covering orbifold} of a Riemannian orbifold $\Orb$ is a Riemannian orbifold $\Orb'$ together with a surjective map $\varphi : \Orb' \To \Orb$ such that each point $x\in \Orb$ has a neighborhood $U$ isometric to some $M/G$ for which each connected component $U_i$ of $\varphi^{-1}(U)$ is isometric to $M/G_i$ for some subgroup $G_i<G$ such that the isometries respect the projections.
\end{dfn}
\begin{figure}
	\centering
		\includegraphics[width=0.35\textwidth]{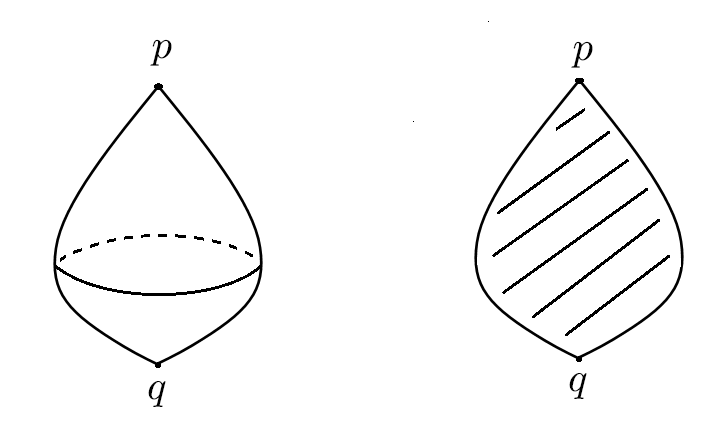}
	\caption{Left: A \emph{$(p,q)$-spindle} orbifold $S^2(p,q)$. Right: A \emph{$(p,q)$-half-spindle} orbifold $D^2(;p,q)$ (no assumptions on $p$ and $q$). $(p,1)$-spindle orbifolds are also known as \emph{teardrops}. The orbifolds in the picture are bad if and only if $p\neq q$.}
	\label{spindle}
\end{figure}
For instance, if a finite group $G$ acts isometrically on a Riemannian $2$-orbifold $\Orb$, then the projection $\Orb \To \Orb /G$ is a covering of Riemannian orbifolds. Thurston showed that the theory of covering spaces (and fundamental groups) works for orbifolds analogously as for manifolds \cite[Sec.~1.2.2]{MR2883685}, see \cite{Thurston} and e.g. \cite[Sec.~2.2]{MR2060653}.

Note that a finite covering orbifold of a Besse orbifold is itself Besse. In particular, the metric double cover of a Besse $2$-orbifold with mirror singularities is a Besse $2$-orbifold with isolated singularities. An orbifold is called \emph{good} (or \emph{developable}) if it is covered by a manifold, and otherwise it is called \emph{bad} \cite{Thurston}. The only bad $2$-orbifolds are depicted in Figure \ref{spindle} \cite[Thm.~2.5]{MR0705527}.

For orbifolds an Euler characteristic can be defined that is multiplicative under coverings and coincides with the usual Euler characteristic in the manifold case \cite{MR2883685}. A $2$-orbifold has positive Euler characteristic if and only if it is either bad or \emph{spherical}, i.e. a quotient of $S^2$ by a finite subgroup of $\Or(3)$. All spherical $2$-orbifolds are listed in Table \ref{tab:gls} for $p=q$ (cf. \cite{MR2883685}). A detailed description of the corresponding finite subgroups of $\Or(3)$ can for instance be found in \cite{LaLi}. Spherical $2$-orbifolds inherit a standard Besse metric from $S^2$. In Proposition \ref{prp:admit_Besse} we will see that a $2$-orbifold admits a Besse metric if and only if its orbifold Euler characteristic is positive. The $2$-orbifolds with positive Euler characteristic are listed in Table \ref{tab:gls}.

\subsection{Besse metrics on $2$-orbifolds} \label{sub:P-metrics_on_spindle}

In \cite{MR2529473} a Besse $(p,q)$-spindle orbifold (cf. Figure \ref{spindle}) is constructed for integers $p$ and $q$ with $\mathrm{gcd}(p,q)=1$ and $1<p<q$ as follows. Consider the action
\[
	\begin{array}{cccl}
		  & S^1 \times S^3				& \LTo  	& S^3 \\
		         & (z,(z_1,z_2)) 	& \LMTo  & (z^p z_1,z^q z_2).
	\end{array}
\]
The quotient $S^3/S^1$ is a Besse $(p,q)$-spindle orbifold. The geodesics on $S^3/S^1$ are precisely the projections of \emph{horizontal geodesics} on $S^3$, i.e. geodesics that are orthogonal to the $S^1$-orbits. In \cite[Thm.~3.6]{MR2529473} the lengths of the geodesic's trajectories on the quotient $S^3/S^1$ are computed. The difference to our result in the case that $p+q$ is odd is due to the fact that the length of a geodesic's trajectory may differ from its length as a parametrized curve as explained in the preceding section. 

Another construction of Besse metrics on $(p,q)$-spindle orbifolds for arbitrary $p$ and $q$ is similar to the construction of non-standard Zoll metrics on $S^2$ due to Tannery and Zoll \cite{MR0496885,MR1511201}. Let $h: [-1,1] \To (- \frac{p+q}{2},\frac{p+q}{2})$ be a smooth, odd function with $h(1)=\frac{p-q}{2}=-h(-1)$ and let a Riemannian metric on $X=(0,\pi)\times ([0,2\pi]/0 \sim 2\pi) \ni (\theta,\phi)$ be defined by
\[
			g=\left( \frac{p+q}{2}+h(\cos(\theta))\right)^2 d\theta^2 + \sin^2(\theta) d\phi^2.
\]
Then the metric completion of $(X,g)$ is a Besse $(p,q)$-spindle orbifold; see \cite[Thm.~4.13]{MR0496885} and note that $p$ and $q$ are defined differently therein. Since the metric is invariant under a reflection in $\phi \in ([0,2\pi]/0 \sim 2\pi)$, it descends to a Besse metric on the corresponding quotient. Hence, we have

\begin{prp}\label{prp:existence_Besse} Every bad $2$-orbifold and every spherical $2$-orbifold, that is, every $2$-orbifold with positive Euler characteristic, admits a Besse metric.
\end{prp}

Another method to construct Besse metrics on $2$-spheres is due to Guillemin \cite{MR0426063}. He shows that for any odd function $\sigma$ on the standard round sphere $(S^2,g_0)$ there exists a
one-parameter family of smooth functions $\rho_t$ on $S^2$ such that $\rho_0=1$, $d\rho_t/dt=\sigma$ at $t=0$ and such that $\mathrm{exp}(\rho_t)g_0$ is a Besse metric for small $t$. Note that if $G< \Or(3)$ does not contain minus the identity, then $\sigma$ can be chosen to be non-trivial and $G$-equivariant.

On spherical orbifolds with more than two isolated singular points only the round Besse metrics seem to be known. However, in view of the result in \cite{MR0426063} we believe in the following 

\begin{conj} Let $G<\Or(3)$. If $S^2/G$ is not covered by the real projective plane, i.e. if $-1 \notin G$, then the moduli space of Besse metrics on $S^2/G$ is infinite-dimensional.
\end{conj}

\subsection{$2$-orbifolds that admit Besse metrics} \label{sub:admitP}

In \cite{MR0400257} Wadsley proves the following result (cf. \cite[Thm.~7.12]{MR0496885}).
\begin{thmw}[Wadsley] If the orbits of a flow on a Riemannian manifold are periodic geodesics parametrized by arc-length, then the flow itself is periodic, so that the orbits have a common period.
\end{thmw}

Using Wadsley's theorem we can prove the following property of Besse $2$-orbifolds.

\begin{prp}\label{prp:Besse_compact} A Besse $2$-orbifold is compact.
\end{prp}
\begin{proof} First suppose that $\Orb$ is a Besse $2$-orbifold with isolated singularities. Then the unit tangent bundle $M=T^1 \Orb$ is an orientable manifold \cite{MR2359514}. It inherits an orientation and a natural Riemannian metric (Sasaki metric) from $\Orb$ \cite[Ch.~1.K]{MR0496885}. With respect to this metric the integral curves of the geodesic field on $M$ are geodesics that project to the geodesics on $\Orb$ with the same arc-length parametrization and the same period. All the geodesics on $\Orb$ can be obtained in this way, ibid. It follows from Wadsley's result that these integral curves on $M$ have a common period, say $l$, and thus so have the geodesics on $\Orb$.

Since $\Orb$ is geodesically complete as an orbifold by the Besse condition, every length minimizing path $\gamma:[0,a)\To \Orb$, which is also an orbifold geodesic, can be extended to a continuous path $\overline{\gamma}:[0,a]\To \Orb$. Therefore $\Orb$ is complete by the Hopf-Rinow theorem \cite[Thm.~2.5.28]{MR1835418} and every pair of points on $\Orb$ can be connected by a (minimizing and hence orbifold) geodesic \cite[Thm.~2.5.23]{MR1835418}. It follows that $\mathrm{diam}(\Orb)\leq l$ and thus that $\Orb$ is compact \cite[Thm.~2.5.28, $(ii)$]{MR1835418}. For a Besse $2$-orbifold $\Orb$ whose singular points are not isolated the same argument applies to its metric double, which is a Besse $2$-orbifold with isolated singularities. Hence, in this case $\Orb$ is compact as the continuous image of its double.
\end{proof}

Now we can prove the following characterization of $2$-orbifolds that admit Besse metrics.

\begin{prp}\label{prp:admit_Besse} A $2$-orbifold admits a Besse metric if and only if it is either bad or spherical, i.e. if and only if its orbifold Euler characteristic is positive.
\end{prp}
\begin{proof} By Proposition \ref{prp:existence_Besse} it remains to prove the only if direction. So let $\Orb$ be a Besse $2$-orbifold. By Proposition \ref{prp:Besse_compact} $\Orb$ is compact. If it is also good, then it is in fact finitely covered by a Besse manifold \cite[Thm.~2.5]{MR0705527}. Since the fundamental group of a Besse manifold is finite \cite[Thm.~7.37]{MR0496885}, $\Orb$ must be spherical in this case.
\end{proof}

In the following three sections we recall some facts that will be needed later. The reader may proceed to Section \ref{sub:orb_of_geo} on first reading and come back to these sections on demand.

\subsection{Almost free circle actions on 3-manifolds and Seifert fiber spaces}
\label{sub:Seif}

Suppose we have a smooth, almost free (i.e. isotropy groups are finite) $S^1$-action on an orientable closed $3$-manifold $M$. Then the orbits are circles and define a decomposition of $M$ into so-called \emph{fibers}. If some element of $S^1$ fixes a point on a fiber, then it fixes the fiber pointwise. A fiber is called \emph{exceptional (or singular) of order $k\geq 2$} if its isotropy subgroup of $S^1$ has order $k$. Since $S^1$ is compact, there exists an $S^1$-invariant Riemannian metric on $M$. The metric quotient $M/S^1$ is an orientable Riemannian $2$-orbifold with isolated singularities and (the orders of) the exceptional fibers correspond to (the orders of) the singularities of $M/S^1$ (cf. \cite{MR2719410}). The manifold $M$ together with a chosen orientation and its decomposition into fibers defines a \emph{Seifert fiber space} of type $o_1$, or Oo in Seifert's original notation \cite{MR1555366}, meaning that both the space and the base are orientable. For the definition of a general Seifert fiber space we refer to \cite{MR1555366}. Roughly speaking it is a closed $3$-manifold together with a decomposition $\mathcal{F}$ into $S^1$-fibers which are, however, only locally defined by an $S^1$-action. Conversely, any Seifert fiber space of type $o_1$ can be obtained in the above way \cite[Ch.~2]{MR0741334}. 
\begin{rem}\label{rem:compt_orientation} The fibers of a Seifert fiber space of type $o_1$ can be oriented in a continuous way. When we speak about orientations of the fibers, we always mean such a continuous choice.
\end{rem}
A Seifert fiber space (of type $o_1)$ is uniquely determined by a finite number of numerical invariants up to orientation- and fiber-preserving diffeomorphism \cite[Thm.~1.5]{MR0741334}. Two sets of invariants determine the same Seifert fiber space, if and only if they are related as described in \cite[Thm.~1.5]{MR0741334}. Forgetting about the orientation of $M$ and allowing general fiber-preserving diffeomorphisms amounts to enlarging the equivalence relation on the set of numerical invariants \cite[Cor.~1.7]{MR0741334}. In \cite[(6.1)]{MR0219086} it is shown that two $S^1$-actions on $M$ define the same Seifert fiber space up to orientation, if and only if there exists a diffeomorphism $h$ of $M$ and an automorphism $a$ of $S^1$ such that for all $m\in M$ and $g \in S^1$ the relation $h(gm)=a(g)h(m)$ holds. In this case it can be shown that this automorphism can in fact be chosen to be the identity \cite[pp.~12-13]{MR0741334}. A specific diffeomorphism that conjugates the $S^1$-actions occurring in this paper to their inverse actions is given in Section \ref{sub:orb_of_geo} (namely by the map $i:T^1 \Orb  \To T^1 \Orb$).  Hence, we have

\begin{lem} \label{lem:seifert_classif}
The classification of smooth, almost free $S^1$-actions on $M$ up to conjugation by diffeomorphisms coincides with the classification of Seifert fiberings on $M$ of type $o_1$ up to orientation.
\end{lem}

A covering of Seifert fiber spaces is a covering that restricts to coverings of fibers on preimages of fibers. 

\begin{rem} \label{rem:lifting} Suppose that $M$ is $k$-foldly covered by $S^3$ via a map $\sigma$ and that $\tau_k : S^1 \To S^1$ is a $k$-fold Lie group covering. Then the action of $S^1$ on $M$ via $\tau_k$ can be lifted to an action on $S^3$, since the map $S^1 \times S^3 \To M$, $(\theta,p) \mapsto \tau_k(\theta)\sigma(p)$ is trivial on fundamental groups. In particular, we see that Seifert fiberings on $M$ of type $o_1$ can be lifted to $S^3$. Moreover, since $\sigma$ preserves the orientations of the fibers induced by the $S^1$-actions, also the group of deck transformations preserves the orientations of the fibers.
\end{rem}

Finally, note that the classifications of Seifert fibered spaces (of type $o_1$) in the topological and the smooth category coincide \cite{Brin,MR0741334,MR1555366}. 

\subsection{Seifert fiberings of lens spaces}\label{sub:lens_spaces} Seifert fiberings of lens spaces are classified in \cite{GL16}. Here we remind of some facts. Recall that for coprime integers $p,q\neq 0$ the $L(p,q)$ \emph{lens space} is defined as a quotient of $S^3=\{(z_1,z_2)\in \Co^2 | |z_1|^2+|z_2|^2=1 \}$ by the free $\Z_p$-action on $S^3$ generated by $e^{2\pi i /p} (z_1,z_2) = (e^{2\pi i /p}z_1,e^{2\pi i q/p}z_2)$. Also recall that two lens spaces $L(p,q)$ and $L(p',q')$ are diffeomorphic, if and only if $p= \pm p'$ and $q \equiv \pm q'^{\pm 1} \text{ mod }p$ \cite[pp.~28-29]{MR0741334}. Given a pair of coprime natural numbers $\alpha_1, \alpha_2$ a Seifert fibering on $S^3$ can be defined by the action $\theta (z_1,z_2) = (e^{i\theta \alpha_1}z_1,e^{i\theta \alpha_2}z_2)$. This Seifert fibering descends to a Seifert fibering of $L(p,q)$ and every Seifert fibering on $L(p,q)$ with orientable base can be obtained in this way \cite[Thm.~5.1.]{GL16}.

An alternative description of lens spaces and Seifert fiberings on them works as follows. Suppose we have two solid tori $T_1$ and $T_2$ and a diffeomorphism $\psi: \partial T_1 \To \partial T_2$.  Then the space $T_1 \cup_{\psi} T_2$ obtained by gluing together $T_1$ and $T_2$ via $\psi$ is a lens space. Moreover, if we choose \emph{meridians} $m_i$ on $T_i$, i.e. embedded loops in $\partial T_i$ that are null-homotopic in $T_i$ and that generate maximal subgroups of $H_1(\partial T_i)$, and \emph{longitudes} $l_i$ on $T_i$, i.e. embedded loops in $\partial T_i$ that generate maximal subgroups of $H_1(T_i)$, and if we have $\psi(m_1) \sim s m_2 + r l_2$ in $H_1(\partial T_2)$, then the space $T_1 \cup_{\psi} T_2$ is a $L(r,-s)$ lens space, see \cite[Thm.~1.3.4.]{Brin} and \cite[Thm.~4.3]{MR0741334}. A standard fibered solid torus is a solid torus $T=D^2 \times S^1$ fibered by the orbits of the almost free $S^1$-action $e^{it}(re^{it_0},e^{it_1})=(re^{it}e^{it_0},e^{ikt}e^{it_1})$ for some positive integer $k$. Every smooth Seifert fibering on a solid torus $T$ is fiber-preservingly diffeomorphic to precisely one of the standard fibered solid tori. Suppose that the solid tori $T_1$ and $T_2$ in the situation above are Seifert fibered and that the diffeomorphism $\psi: \partial T_1 \To \partial T_2$ preserves fibers. Then we obtain an induced Seifert fibering of the lens space $T_1 \cup_{\psi} T_2$. Moreover, if $m_1$ is a meridian of $T_1$, then the fiber-homeomorphism (and hence fiber-diffeomorphism) type of $T_1 \cup_{\psi} T_2$ is completely determined by the homology class of $\psi(m_1)$ in $H_1(\partial T_2)$ \cite[Thm.~1.3.4.]{Brin}.

\subsection{Finite group actions on $S^1$ and $S^2$}\label{sub:linear}
We will encounter isometric actions of finite groups on Riemannian $2$-orbifolds $\Orb$ with $|\Orb|=S^2$. Such an action can be smoothed, i.e. there exists a smooth structure on $|\Orb|$ with respect to which the group acts smoothly. Indeed, the orbifold admits an equivariant triangulation and the corresponding simplicial complex can be equivariantly smoothed \cite{Lange3}. Then it follows from the classification of $2$-orbifolds with positive Euler characteristic (see \cite{MR2883685}) that the action can be conjugated to a linear action on $S^2$ (cf. \cite{MR2954692}).

In a similar way one can show that a continuous action of a finite group on a circle can be conjugated by an orientation-preserving homeomorphism to a linear action. Hence, if such an action preserves the orientation of the circle, then it must be cyclic. Moreover, if the order of an orientation-preserving homeomorphism $h$ of the circle $S^1$ is at least $3$ but finite, then its linearized action rotates the circle about an angle different from $\pi$. In this situation we say that the \emph{circle is rotated in a positive or negative direction} with respect to a chosen orientation if the angle rotated by the linearized action (obtained through conjugation by an \emph{orientation-preserving} homeomorphism) measured with respect to the chosen orientation is smaller or greater than $\pi$. Observe that $h$ rotates $S^1$ in a positive direction if and only if, following $S^1$ from $x$ in the positive direction, one \emph{encounters} $h(x)$ before $h^{2}(x)$. In particular, in this way it makes sense to say that a homeomorphism rotates two fibers of a Seifert fiber space of type $o_1$ in the \emph{same or in different directions}, cf. Remark \ref{rem:compt_orientation}. We will need the following statement.

\begin{lem}\label{lem:rot_same_direction} Let $(M,\mathcal{F})$ be a connected Seifert fiber space of type $o_1$ with only regular fibers. Let $h$ be a homeomorphism of $M$ of finite order $n\geq 3$ that leaves all fibers invariant and preserves their orientation. Suppose that the restriction of $h$ to each fiber has order $n$. Then $h$ rotates all fibers in the same direction. 
\end{lem}
\begin{proof}
By the connectedness assumption it suffices to prove the conclusion for a standard fibered solid torus with only regular fibers, i.e. for $D^2 \times S^1$. In this case the conclusion follows from continuity and the ``first encounter criterion'' above by looking at the orbits of a section $S=D^2\times \{*\}$ of the fibered solid torus under $h$. Indeed, if two fibers were rotated in different directions, then the images $h(S)$ and $h^2(S)$ of $S$ would have a nontrivial intersection by the first encounter criterion, resulting in a fixed point of $h$. This can only happen if $h$ is the identity, in contradiction to our assumption on the order of $h$.
\end{proof}
\begin{rem} It can be shown that the assumption in the lemma on the orders of the restrictions of $h$ to the fibers actually follows, too. Moreover, these conclusions still hold for the regular fibers of a general Seifert fiber space of type $o_1$ with a homeomorphism $h$ as in the lemma. However, all our assumptions will be satisfied in our application in Lemma \ref{lem:no_commute}.
\end{rem}

\subsection{Orbifolds of geodesics and geodesic periods}
\label{sub:orb_of_geo}
In the following a geodesic on a Besse orbifold is supposed to be prime unless stated otherwise. Suppose that $\Orb$ is a Besse $2$-orbifold with only isolated singularities. This is in particular the case if $\Orb$ is orientable. Then $M=T^1 \Orb$ is a manifold and the geodesic flow on it is periodic due to theorems by Epstein \cite{MR0288785} or Wadsley \cite{MR0400257} as discussed in Section \ref{sub:admitP}. Hence, the flow defines a Seifert fibering $\mathcal{F}_g$ of type $o_1$ on $M=T^1 \Orb$ whose fibers inherit a natural orientation from the flow. The quotient $\Orb_g=M/\mathcal{F}_g$ parametrizes the closed orbits of the geodesic flow on $M=T^1 \Orb$ and has a natural orbifold structure with isolated singularities, see Section \ref{sub:Seif}. A point on $\Orb_g$ of order $k$ corresponds to (an equivalence class of reparametrizations of) a geodesic on $\Orb$ whose period is $k$-times shorter than the period of a generic geodesic on $\Orb$. We say that this geodesics is of \emph{order $k$} and we call it \emph{exceptional} if $k>1$ and \emph{regular} otherwise. We will see that in this situation $\Orb_g$ is always a quotient of $S^2$ by a finite subgroup of $\SOr(3)$ and as such a topological sphere \cite{Lange2}.

A non-orientable Besse $2$-orbifold $\Orb$ is a metric quotient of an orientable Besse $2$-orbifold $\hat{\Orb}$ with isolated singularities by an isometric orientation reversing involution $s$. Since the $S^1$-action on $T^1 \hat\Orb$ defining $\mathcal{F}_g$ is normalized by $s$, the auxiliary metric on $T^1 \hat\Orb$ can be chosen to be $s$ invariant (cf. Section \ref{sub:Seif}) so that $s$ induces an orientation preserving isometry $s: \hat{\Orb}_g \To \hat{\Orb}_g$. In particular, $s$ either acts trivially on $\hat{\Orb}_g$ or rotates $\hat{\Orb}_g$ about two points (cf. Section \ref{sub:linear}). We have $T^1 \Orb=T^1 \hat{\Orb}/s$ (cf. \cite{MR2359514}) and set $\Orb_g:=\hat{\Orb}_g/s$, which is a Riemannian orbifold. Since $\hat \Orb_g$ will always be topologically a sphere, so will be $\Orb_g$. The points on $\Orb_g$ still correspond to (equivalences classes of reparametrizations of) geodesics on $\Orb$. However, if the singularities of $\Orb$ are not isolated, then $T^1 \Orb$ has orbifold singularities that cause singularities on $\Orb_g$ whose orders do not directly correspond to the periods of the respective geodesics on $\Orb$. In this situation the following lemma provides information on the periods of geodesics on $\Orb$.

\begin{lem}\label{lem:pro_geo_per} Suppose that $\hat\Orb \To \Orb$ is the orientable double cover of a non-orientable Besse $2$-orbifold $\Orb$ and that $s:\hat\Orb \To \hat\Orb$ generates the group of deck transformations. Then a prime geodesic on $\hat \Orb$ projects to a geodesic on $\Orb$. This geodesic on $\Orb$ is not prime if and only if the geodesic on $\hat \Orb$ is fixed by $s$ as an element of $\hat{\Orb}_g$ but not pointwise fixed as a fiber of $\mathcal{F}_g$. In this case the corresponding prime geodesic on $\Orb$ is covered twice by the geodesic on $\hat \Orb$.
\end{lem}
\begin{proof} The projection of a geodesic on $\hat \Orb$ to $\Orb$ is a geodesic by definition, since a small neighborhood in $\Orb$, isometric to some $M/G$ as in Definition \ref{dfn:orbifold}, is covered by a small neighborhood in $\hat \Orb$ isometric to $M/H$ for an index $2$ subgroup $H$ of $G$. A closed geodesic $\gamma$ (say of period 1) is prime if and only if for each positive integer $n$ and some (and then all) $t\in[0,1-1/n]$ we have $\gamma'(t)\neq \gamma'(t+1/n)$ as elements of the unit tangent bundle. A prime geodesic $\gamma$ on $\hat \Orb$ (of period 1) is invariant under $s$ (i.e. fixed as an element of $\hat \Orb_g$) if and only if $s(\gamma'(0))=\gamma'(t_0)$ for some $t_0\in [0,1]$ (and then also for some $t_0 \in \{0,1/2\}$ since $s$ has order $2$). It is pointwise fixed by $s$ if and only if for some (and then each) $t\in [0,1]$ we have $s(\gamma'(t))=\gamma'(t)$. Putting this together proves the second claim. In this case we have $s(\gamma'(0))=\gamma'(1/2)\neq \gamma'(0)$ and so the last claim follows, too.
\end{proof}

Examples for the two possible cases in the lemma are given by a reflection and an inversion of $S^2$. In the first case geodesics in the fixed point set of $s$ are both $s$-invariant and pointwise fixed by $s$. In the second case every geodesic is $s$-invariant but not pointwise fixed by $s$. 

The orbifolds $\Orb_g$ obtained in this way from Besse $2$-orbifolds have the following additional symmetry. Consider the involution $i:T^1 \Orb  \To T^1 \Orb$ mapping $(x,v)$ to $(x,-v)$. The involution $i$ is orientation-preserving and interchanges fibers of $\mathcal{F}_g$ representing different orientations of the same geodesic trajectory on $\Orb$ while reversing their natural orientation. We can suppose that the auxiliary Riemannian metric on $T^1 \Orb$ is also invariant under $i$ (cf. preceding paragraph and Section \ref{sub:Seif}). Then $i$ induces an involutive orientation-reversing isometry of $\Orb_g$ with the quotient being a Riemannian orbifold. In particular, it maps singular points to singular points of the same order. We introduce the following concept; cf. \cite[2.5.]{MR0496885} for the manifold analogue.

\begin{dfn} \label{dfn:orbifold_of_geodesics} For a Besse $2$-orbifold $\Orb$ we define the \emph{orbifold of oriented geodesics} to be $\Orb_g$ and the \emph{orbifold of non-oriented geodesics} to be $\Orb_g/i$.
\end{dfn}

Recall that we sometimes view geodesics on $\Orb$ as points on $\Orb_g$, that is we forget about the specific reparametrization. We distinguish two kinds of geodesics on $\Orb$. 

\begin{dfn} \label{dfn:selfinverse_of_geodesics} We call a geodesic on $\Orb$ \emph{self-inverse} if it is a branch point of the covering $\Orb_g \To \Orb_g/i$, that is, if it is fixed by $i$ as a point on $\Orb_g$ .
\end{dfn}

The following statement is a consequence of the discussion after Definition \ref{dfn:periods}.

\begin{lem}\label{lem:fix_point_i}
A geodesic on a Besse $2$-orbifold $\Orb$ is self-inverse if and only if it hits an isolated singularity of even order on $\Orb$ or the boundary of $|\Orb|$ perpendicularly (in the sense of centrically at the corner reflector singularities of odd order, cf. Section \ref{sub:Riem_orb}). In particular, if $\Orb_g$ is topologically a sphere, then the orbifold of non-oriented geodesics $\Orb_g/i$ is topologically a disk in the case that $\Orb$ has isolated singularities of even order or a topological boundary, and otherwise a projective plane.
\end{lem}
\begin{proof} By definition a geodesic $\gamma$ on $\Orb$ (say of period 1) is self-inverse if it is fixed by $i$ as an element of $\Orb_g$. This is precisely the case if $-\gamma'(0)=\gamma'(t_0)$ for some $t_0\in [0,1]$. In this case we have $-\gamma'(t_0/2)=\gamma'(t_0/2)$ in $T_{\gamma(t_0)} \Orb$. This happens if and only if $\gamma$ hits a singular point of $\Orb$ at $\gamma(t_0)$ as described in the lemma. For the second claim recall that $i$ reverses the orientation of $\Orb_g\cong S^2$ and is thus either conjugated to a reflection or the inversion by Section \ref{sub:linear}. Hence, if $i$ has a fixed point it is conjugated to a reflection and the corresponding quotient is a disk. Otherwise, it is conjugated to the inversion and the quotient is a projective plane.
\end{proof}

By the \emph{trajectory} of a geodesic we mean its geometric image in $\Orb$. The trajectories of geodesics are in one-to-one correspondence with the geodesic's images in $\Orb_g/i$. Since the period of a geodesic is $i$-invariant, we can talk about periods of geodesic trajectories. By the \emph{(non)-self-inverse geodesic periods} of $\Orb$ we mean the set of periods of (non)-self-inverse geodesic trajectories on $\Orb$ counted with multiplicity.

\begin{dfn} \label{dfn:period_spectrum} By the \emph{(labeled) geodesic periods} of a Besse $2$-orbifold $\Orb$ we mean the data encoded in its self-inverse and its non-self-inverse geodesic periods.
\end{dfn}

In the following we suppose that all Besse $2$-orbifold are normalized such that their maximal geodesic period is one. We will see that for a Besse $2$-orbifold $\Orb$ almost all geodesic trajectories on $\Orb$ have the same regular maximal period and that the periods of the finitely many exceptions are integer factors of this regular period. Suppose the periods of the exceptional non-self-inverse geodesic trajectories are specified by the integer factors $k_1,\ldots,k_l$ and the periods of the exceptional self-inverse geodesic trajectories are specified by the integer factors $k'_1,\ldots,k'_{l'}$. We will also see that there always exists a geodesic of maximal period which is non-self-inverse and that there are either infinitely many or no geodesics of maximal period that are self-inverse. In this case, the geodesic periods of $\Orb$ can be recorded symbolically as the labeled unordered tuple $(1_{\infty},\overline{k_1},\ldots,\overline{k_l},k'_1,\ldots,k'_{l'})$ or $(\overline{k_1},\ldots,\overline{k_l},k'_1,\ldots,k'_{l'})$ depending on whether there exist geodesics of maximal period that are self-inverse or not. For ease of parlance we agree upon calling this tuple the geodesic periods of $\Orb$.

To summarize, in case of a Besse $2$-orbifold with isolated singularities the geodesic periods and the data encoded in the covering $\Orb_g \To \Orb_g/i$ determine each other. In fact, the map $i$ acts as a reflection or an inversion on the topological sphere $\Orb_g$ depending on whether $1_{\infty}$ occurs in the geodesic periods or not, the pairs of singularities on $\Orb_g$ interchanged by $i$ correspond to the $\overline{k_1},\ldots,\overline{k_l}$, and the singularities on $\Orb_g$ in the fixed point set of $i$ correspond to the $k'_1,\ldots,k'_{l'}$. However, in general these data do not determine each other. For instance, $S^2(2,2)$ and $D^2(4;)$ with some Besse metric have the same geodesic periods, but different orbifolds of oriented and non-oriented geodesics. Also, the orbifolds of oriented and unoriented geodesics of $D^2(;2)$ and $D^2(;4,2)$, or of $S^2(2,3,4)$ and $D^2(2,3,4)$ with some Besse metric coincide, but (even) their (``unlabeled'') geodesic periods differ (see Section \ref{sub:nonorbad}). Hence, in order to prove rigidity of the geodesic periods in case of an orbifold with non-isolated singularities, a more detailed analysis involving geometric arguments based on  Lemma \ref{lem:pro_geo_per} will be necessary; see Section \ref{sub:nonorbad} and Section \ref{sub:non-orient_spher}.


\section{Proof of the main result}
\label{sec:proof}

In this section we show our main results on geodesic periods and orbifolds of geodesics of Besse $2$-orbifolds. We first treat the case of spindle orbifolds which, together with the case of the real projective plane, forms a central part of our proof.

\subsection{Spindle orbifolds} \label{sub:spindle}

Let $\Orb$  be a Besse $(p,q)$-spindle orbifold. Recall that we do not make assumptions on $\mathrm{gcd}(p,q)$. We claim that the orbifold of geodesics $\Orb_g$ and $\Orb_g/i$ and the geodesic periods are given as stated in Table \ref{tab:gls} in the appendix. In particular, we show that $\Orb_g=S^2((p+q)/\kappa,(p+q)/\kappa)$ where $\kappa$ is $1$ or $2$ depending on whether $p+q$ is odd or even. The proof is divided into steps a)-f).
\newline
\newline
\noindent a) Recall from Section \ref{sub:orb_of_geo} that the unit tangent bundle $M=T^1 \Orb$ is a manifold.

\begin{lem}\label{lem:gluing} The unit tangent bundle $M=T^1 \Orb$ of a $(p,q)$-spindle orbifold $\Orb$ is a lens space. More precisely, we have $M\cong L(p+q,1)$.
\end{lem}
\begin{figure}
	\centering
		\includegraphics[width=0.48\textwidth]{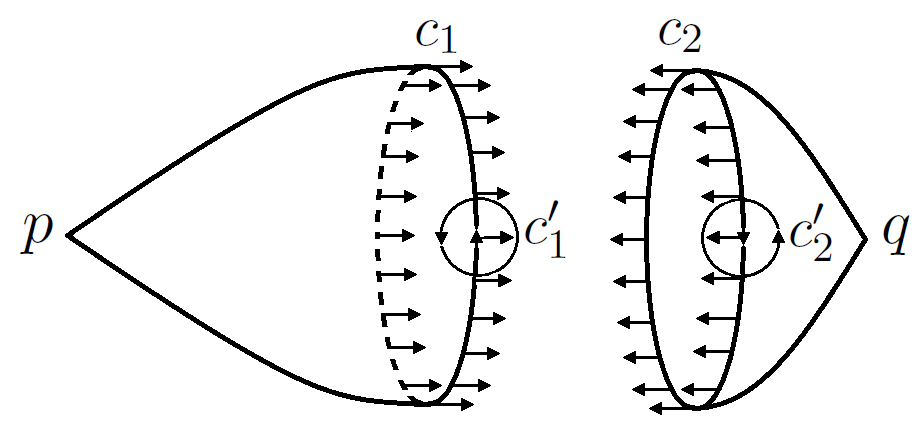}
	\caption{Gluing construction in Lemma \ref{lem:gluing}. Note that the curve $c_1$ is homotopic to the curve $c_2^{-1}$, that is, $c_2$ traversed in the opposite direction.}
	\label{fig:gluing}
\end{figure}
\begin{proof} To prove the lemma we choose an equator of $\Orb$ that separates the two singular points such that $\Orb$ decomposes into two closed disks $D_i$ contained in open sets $U_i$, $i=1,2$, each of which admits an orbifold chart $\tilde U_i$. We denote the cyclic group acting on $\tilde U_i$ by $\Gamma_i$ and the preimage of $D_i$ in $\tilde U_i$ by $\tilde D_i$, $i=1,2$. The space $M$ decomposes accordingly into the preimages $T_i$ of the disks $D_i$ that are contained in the open preimages $V_i$ of the $U_i$. Let $\tilde V_i$ be the covering chart of $V_i$ induced by the chart $\tilde U_i$ and let $\tilde T_i$ be the preimage of $T_i$ in $\tilde V_i$. Then $\tilde T_i$ is the restriction of the unit tangent bundle of $\tilde U_i$ to the disk $\tilde D_i$ and as such a full torus. We choose a smooth identification of $\tilde U_i$ with an open ball in $\R^2$ with respect to which the action of $\Gamma_i$ becomes linear, and which restricts to a diffeomorphism  $D^2 \To \tilde D_i$. This induces a diffeomorphism $D^2 \times S^1 \To \tilde T_i$ that maps the $S^1$-fibers to the fibers of the foot-point projection $\tilde \pi_i : \tilde T_i \To \tilde D_i$, $D^2$ to a section of $\tilde \pi_i$, and with respect to which $\Gamma_i$ acts diagonally in both factors by rotations. In particular, the $T_i$ are full tori themselves. The space $M$ can be recovered from these full tori by a specification of the gluing homeomorphism $\psi: \partial T_1 \To \partial T_2$. The homeomorphism type of $T_1 \cup_{\psi} T_2$ is determined by the homology class of $\psi(m_1)$ in $\partial T_2$ for a meridian $m_1$ of $\partial T_1$, i.e. an embedded loop in $\partial T_1$ that is null-homotopic in $T_1$ and that generates a maximal subgroup of $H_1(\partial T_1)$ (cf. Section \ref{sub:lens_spaces}). With respect to the splitting above we define two loops $\tilde{c}_1,\tilde{c}'_1 : S^1 \To \partial \tilde{T}_1 = S^1 \times S^1$ by $\tilde{c}_1(z)=(z,z)$ and $\tilde{c}'_1(z)=(1,z)$ and two loops $\tilde{c}_2,\tilde{c}'_2 : S^1 \To \partial \tilde{T}_2$ analogously. Note that the orientations of the fibers are chosen in such a way that the loops $\tilde{c}_i$ are invariant under $\Gamma_i$. We can choose meridians $\tilde{m}_i$ of $\tilde{T}_i$ such that in homology of $\partial \tilde{T}_i$ we have $\tilde{m}_i \sim -\tilde{c}_i+\tilde{c}'_i$. The meridians $\tilde{m}_i$ of $\tilde{T}_i$ project to meridians $m_i$ of $T_i$. Let $c_i': S^1 \To \partial T_i$ be the projection of $\tilde{c}'_i$ and let $c_i: S^1 \To \partial T_i$ be curves such that in homology $\tilde{c}_i$ projects to $r_i \cdot c_i$, where $r_1=p$ and $r_2=q$. Then, in homology we have $m_i \sim -r_i c_i + c'_i$. Observe that we recover $M$ if the attaching map $\psi$ satisfies $\psi(c_1) \sim - c_2$ and $\psi(c'_1) \sim c'_2$, see Figure \ref{fig:gluing}. Picking $l_2=c_2$ as a longitude in $\partial T_2$, i.e. an embedded loop in $\partial T_2$ that generates a maximal subgroup of $H_1(T_2)$ (cf. Section \ref{sub:lens_spaces}), we have $\psi(m_1) \sim p c_2 + c'_2 = 1 \cdot m_2+(p+q) \cdot  l_2$. The resulting space $T_1 \cup_{\psi} T_2$ is a lens space of type $L(p+q,-1)\cong L(p+q,1)$ as claimed (cf. Section \ref{sub:lens_spaces}).
\end{proof}

\noindent b) Recall from Section \ref{sub:orb_of_geo} that $\mathcal{F}_g$ denotes the Seifert fibering of type $o_1$ on $M=T^1 \Orb$ defined by the geodesic flow. As an auxiliary tool we also need the Seifert fibering $\mathcal{F}_t$ on $M$ induced by the projection $M=T^1 \Orb \To \Orb$. Since both $M$ and $\Orb$ are orientable, this Seifert fibering is of type $o_1$, too (cf. Section \ref{sub:Seif}). We have $M/\mathcal{F}_t =\Orb$ as a Riemannian orbifold. By Remark \ref{rem:lifting} the Seifert fiberings $\mathcal{F}_t$ and $\mathcal{F}_g$ can be lifted to Seifert fiberings of type $o_1$ of the universal cover $\tilde M$ of $M$. We denote these lifts by $\tilde{\mathcal{F}}_t$ and $\tilde{\mathcal{F}}_g$, and the corresponding quotients by $\tilde \Orb_t:=\tilde M/\tilde{\mathcal{F}}_t$ and $\tilde \Orb_g:=\tilde M/\tilde{\mathcal{F}}_g$. The fiberings $\mathcal{F}_t$ and $\mathcal{F}_g$ on $M$ as well as their lifts $\tilde{\mathcal{F}}_t$ and $\tilde{\mathcal{F}}_g$ on $\tilde{M}$ are fiberwise transversal.
\newline
\newline
\noindent c) We have the following commutative diagram
\[
	\begin{xy}
		\xymatrix
		{
		  \tilde{\Orb}_t \ar[d] && \tilde{M}\cong S^3 \ar[ll] \ar[rr] \ar[d] && \tilde{\Orb}_g \ar[d]  \\
		 \Orb\cong S^2(p,q)  &&M\cong L(p+q,1) \ar[ll]  \ar[rr]  && \Orb_g  
		}
	\end{xy}
\]	
where the outer vertical projections are coverings of Riemannian orbifolds. The upper horizontal projections induce surjections $\pi_1(\tilde{M}) \To \pi_1^{\mathrm{orb}}(\tilde{\Orb_t})$ and $\pi_1(\tilde{M}) \To \pi_1^{\mathrm{orb}}(\tilde{\Orb_g})$ \cite[Lem.~3.2]{MR0705527}. According to the classification of simply connected, compact $2$-orbifolds \cite[Thm.~2.5]{MR0705527} this implies $\tilde{\Orb}_g\cong S^2(p_g,q_g)$ and $\tilde{\Orb}_t\cong S^2(p_t,q_t)$ as orbifolds for coprime $p_g,q_g$ and coprime $p_t,q_t$. In other words, Seifert fiberings on $S^3$ are uniquely determined by two coprime positive integers \cite{MR1555366} (cf. \cite[Prop.~5.2.]{GL16}). Let $\Gamma \cong \Z_{p+q}$ be the group of deck transformations of the covering $\tilde{M} \To M$. The action of $\Gamma$ on $\tilde{M}$ induces actions on $\tilde{\Orb}_t$ and $\tilde{\Orb}_g$ which are not necessarily effective. We denote by $\Gamma_t$ and $\Gamma_g$ the quotients of $\Gamma$ by the respective kernels. The groups $\Gamma_t$ and $\Gamma_g$ are cyclic and we have $\tilde{\Orb}_t/\Gamma_t=\Orb$ and $\tilde{\Orb}_g/ \Gamma_g= \Orb_g$. Since the group $\Gamma$ preserves the orientations of the fibers of $\tilde{\mathcal{F}}_t$ and $\tilde{\mathcal{F}}_g$ (cf. Remark \ref{rem:lifting}) and the orientation of $\tilde{M}$, the groups $\Gamma_t$ and $\Gamma_g$ preserve the orientation of $|\tilde{\Orb}_t|\cong S^2$ and $|\tilde{\Orb}_g|\cong S^2$, respectively. Therefore, the action of $\Gamma_t$ and of $\Gamma_g$ can be conjugated to a standard action of a cyclic group on $S^2$ (cf. Section \ref{sub:linear}). Moreover, since $\Gamma$ acts isometrically on $\tilde{\Orb}_t$ and $\tilde{\Orb}_g$ with respect to the orbifold metrics introduced above (cf. step b) and Section \ref{sub:Seif}), it fixes the singular points. Both together implies $p=|\Gamma_t|p_t$, $q=|\Gamma_t|q_t$ (up to permutation) and $\Orb_g=S^2(|\Gamma_g|p_g,|\Gamma_g|q_g)$.
\begin{rem}\label{rem:Lusternik_Schnirelmann}
In the case of $\Orb \cong S^2$ in \cite{MR0636885} the theorem of Lusternik-Schnirelmann, which guarantees the existence of three simple closed geodesics, is applied at this point to show that there exists a simple regular geodesic. It is then not difficult to conclude that all geodesics are simple and regular; see \cite{MR0636885}.
\end{rem}
\noindent d)  Let $\tilde{i}:\tilde{M} \To \tilde{M}$ be a lift of the involution $i:M \To M$ introduced in Section \ref{sub:orb_of_geo}.
\begin{lem} \label{lem:orientation_i_tilde}The lift $\tilde{i}$ preserves the orientations of the fibers of $\tilde{\mathcal{F}}_t$ while it reverses the orientations of the fibers of $\tilde{\mathcal{F}}_g$.
\end{lem}
\begin{proof} The claim follows from the respective property of the action of $i$ on the fibers of $\mathcal{F}_t$ and $\mathcal{F}_g$.
\end{proof}

If $pq$ is odd,  then $i: \Orb_g \To \Orb_g$ has no fixed points by Lemma \ref{lem:fix_point_i}. Hence, in this case it follows that $|\Gamma_g|p_g=|\Gamma_g|q_g$, and thus that $p_g=q_g=1$ since $p_g$ and $q_g$ are coprime. For even $pq$ the same conclusion will follow from the subsequent lemma. 
\begin{lem} \label{lem:commute}The actions of $\tilde{i}$ and $\Gamma$ on $\tilde{M}\cong S^3$ commute.
\end{lem} 
\begin{proof} Let $S^1_t$ be a $\Gamma$-invariant fiber of $\tilde{\mathcal{F}}_t$ and let $\gamma \in \Gamma$ be nontrivial. Since $i$ leaves the fibers of $\mathcal{F}_t$ invariant, its lift $\tilde{i}$ leaves preimages of $\mathcal{F}_t$-fibers invariant. In particular, we see that $\tilde i$ leaves $S^1_t$ invariant. The orientation of $S^1_t$ is preserved by $\gamma$ due to Remark \ref{rem:lifting} and by $\tilde i$ due to Lemma \ref{lem:orientation_i_tilde}. As a lift of $i$ the map $\tilde{i}$ normalizes $\Gamma$. Therefore the restrictions of $\gamma$ and $\tilde{i}$ to $S^1_t$ generate a finite group that preserves the orientation of $S^1_t$. Since the action can be conjugated to a linear (orientation-preserving) action on a circle (cf. Section \ref{sub:linear}), we see that the generated group must be cyclic and thus $\gamma$ and $\tilde{i}$ commute on $S^1_t$. The fact that their commutator $\gamma\tilde{i}\gamma^{-1}\tilde{i}^{-1}: \tilde{M}\To \tilde{M}$ is a lift of the identity of $M$, implies that $\gamma$ and $\tilde{i}$ commute everywhere. Hence, the claim follows.
\end{proof}
Indeed, now we can show
\begin{lem}\label{lem:exceptional_geo} The involution $i$ does not fix singular points on $\Orb_g$. In particular, geodesics hitting a singularity on $\Orb$ of even order are regular.
\end{lem}
\begin{proof} By Lemma \ref{lem:fix_point_i} we only need to consider the case that $pq$ is even. Suppose a singular point on $\Orb_g$ is fixed by $i$. We have seen in c) that this singular point has a single preimage in $\tilde \Orb_g$ whose corresponding fiber $S^1_c$ of $\tilde{\mathcal{F}}_g$ and its orientation are $\Gamma$-invariant. Therefore the map $\tilde{i}$ also leaves the fiber $S^1_c$ invariant (cf. proof of Lemma \ref{lem:commute}) but it reverses its orientation by Lemma \ref{lem:orientation_i_tilde}. Moreover, as in the proof of Lemma \ref{lem:commute} we see that $\Gamma$ and $\tilde{i}$ generate a finite group acting on $S^1_c$, and that the action can be assumed to be linear. The fact that the orientation is not preserved by this action implies that the generated group must be a dihedral group. Since $\Gamma$ and $\tilde{i}$ commute by Lemma \ref{lem:commute}, this dihedral group must be abelian. The only abelian dihedral groups have order $2$ and $4$, respectively. Since the action of $\Gamma$ on $S^1_c$ is effective (as a restriction of a deck transformation action), it follows that $2\leq p+q=|\Gamma|\leq 2$. This contradicts the existence of a singular point on $\Orb$ of even order and thus the first claim follows. Now, the second claim is a consequence of Lemma \ref{lem:fix_point_i}.
\end{proof}
Consequently, in any case we have $\Orb_g=S^2(|\Gamma_g|,|\Gamma_g|)$ and so it remains to determine the order of $\Gamma_g$.
\newline
\newline
\noindent e)
An example of a Seifert fibering on $S^3$ is the Hopf fibration $\mathcal{H}$ defined by the free $S^1$-action $\varphi^+$ (or $\varphi^-$)
\[
	\begin{array}{cccl}
		  \varphi^{\pm}: & S^1 \times S^3				& \LTo  	& S^3 \\
		         & (e^{it},(z_1,z_2)) 	& \LMTo  & (e^{it}z_1,e^{\pm it}z_2).
	\end{array}
\]
Since the actions of $\varphi^{+}$ and $\varphi^{-}$ commute, they induce almost free actions $\varphi^{\pm}: S^1 \times L(r,1)	\LTo L(r,1)$ and Seifert fiberings $\mathcal{F}^{\pm}$ on $L(r,1)$ where $L(r,1)$ is the quotient of $S^3$ by the action of $\varphi^{+}$ restricted to the $r$-th roots of unity in $S^1$. The following lemma shows that these are the only actions that can occur in our situation up to conjugation. For a systematic classification of Seifert fiberings of lens spaces we refer the reader to \cite{GL16}. Given Lemma \ref{lem:seifert_classif} the following lemma is contained therein as a special case \cite[Example~4.17 and Section 5]{GL16}. Here we give a short, self-contained proof.

\begin{lem}\label{lem:lens} Let $\varphi :  S^1 \times L(r,1) \LTo L(r,1)$, $r>1$, be a smooth, almost free action with quotient orbifold of type $S^2(k,k)$ for some positive integer $k$. If there are no exceptional fibers, i.e. if $k=1$, then $\varphi$ is smoothly conjugated to $\varphi^+$. If $k>1$, then $\varphi$ is smoothly conjugated to $\varphi^-$. In the latter case we have $r=\kappa k$, where $\kappa$ is $1$ or $2$ depending on whether $r$ is odd or even. In particular, $r$ is divisible by $4$ if $k$ is even. 
\end{lem}
\begin{rem} In the situation of the lemma the quotient orbifold is actually always a spindle orbifold. The only real assumption is that the orders of the singularities coincide.
\end{rem}
\begin{proof}
Let $\mathcal{F}$ be the Seifert fibering defined by $\varphi$. By Lemma \ref{lem:seifert_classif} it is sufficient to show that the Seifert invariants of $\mathcal{F}$ coincide up to orientation with either those of $\mathcal{F}^{+}$ or $\mathcal{F}^{-}$.

Since $L(r,1)/\mathcal{F}$ is a $S^2(k,k)$ orbifold by assumption, the Seifert fiber space $(L(r,1),\mathcal{F})$ can be obtained as $T_1 \cup_{\psi} T_2$ where $T_1$ and $T_2$ are fibered solid tori with an exceptional fiber of order $k$ and where $\psi: \partial T_1 \To \partial T_2$ is a fiber-preserving homeomorphism \cite[Thm.~1.4.5.]{Brin}.  Choose meridians $m_1$ and $m_2$ on $\partial T_1$ and $\partial T_2$ (cf. Section \ref{sub:lens_spaces}) and a longitude $l_2$ on $\partial T_2$ (cf. Section \ref{sub:lens_spaces}). In homology we have $\psi(m_1) \sim s' m_2 + r' l_2$ for some integers $s',r'$. Since $T_1 \cup_{\psi} T_2$ is a $L(r,1)$ lens space with $r>1$ by assumption, we have $r' = \varepsilon_1 r$ and $s'= \varepsilon_2 + t r\neq 0$ for some integer $t$ and some $\varepsilon_1,\varepsilon_2 \in \{ \pm 1 \}$ (cf. Section \ref{sub:lens_spaces}). Replacing $l_2$ by $tm_2+ \varepsilon_1 l_2$ and $m_2$ by $\varepsilon_2 m_2$ we can assume that $\psi(m_1) \sim m_2 + r l_2$. In particular, $l_1= \psi^{-1}(l_2)$ is a longitude of $\partial T_2$. Let $f_2$ be a regular fiber on $\partial T_2$. Possibly after reversing the orientation of $f_2$ we have $f_2 \sim b_2 m_2 + k l_2$ for some integer $b_2$. The preimage $f_1= \psi^{-1}(f_2)$ is a regular fiber on $\partial T_1$ and we have $f_1 \sim b_1 m_1 + \varepsilon k l_1$ for some integer $b_1$ and some $\varepsilon \in \{ \pm 1 \}$. Since $f_1$ and $f_2$ are without self-intersections, we have $\mathrm{gcd}(b_1,k)=\mathrm{gcd}(b_2,k)=1$. Now the conditions $\psi(m_1) \sim m_2 + r l_2$, $\psi(l_1) \sim l_2$ and $\psi(f_1)=f_2$ imply that $b_1=b_2$ and $rb_1=k(1- \varepsilon)$. Hence, because of $\mathrm{gcd}(b_1,k)=1$, we are in one of the following three mutually exclusive cases
\begin{enumerate}
\item $\varepsilon=1$, $k=1$, $b_1=b_2=0$.
\item $\varepsilon=-1$, $2k=r$ even, and  $b_1=b_2=1$.
\item $\varepsilon=-1$, $k=r$ odd, and  $b_1=b_2=2$.
\end{enumerate}
Since these data completely determine the fiber-homomorphism type of $T_1 \cup_{\psi} T_2$ and since the same argument applies to $(L(r,1),\mathcal{F}^{\pm})$ the claim follows.
\end{proof}
More specifically, a computation shows (see \cite[Example~4.17]{GL16}) that the Seifert invariants of the Seifert fibered spaces occurring in the lemma are given as follows.
\begin{enumerate}
\item $(L(r,1),\mathcal{F})=M(0;(1,r))=(L(r,1),\mathcal{F}^+)$ with $k=1$
\item $(L(r,1),\mathcal{F})=M(0;(k,1),(k,1))=(L(r,1),\mathcal{F}^-)$ with $r=2k$ even
\item $(L(r,1),\mathcal{F})=M(0;(k,\frac{1+k}{2}),(k,\frac{1-k}{2}))=(L(r,1),\mathcal{F}^-)$ with $r=k$ odd.
\end{enumerate}
\noindent f) In case of a sphere, i.e. when $p=q=1$, we have $|\Gamma_g| \in\{1,2\}$ and thus $|\Gamma_g|=1$ and $\Orb_g=S^2$ by the last claim of Lemma \ref{lem:lens}. Suppose that $p+q>2$. According to Lemma \ref{lem:lens} $(M,\mathcal{F}_g)$ is either smoothly conjugated to $(L(p+q,1),\mathcal{F}^+)$ or to $(L(p+q,1),\mathcal{F}^-)$. Hence, we can assume that $\tilde{\mathcal{F}}_g$ is the Hopf fibration $\mathcal{H}$ on $S^3$ defined by $\varphi^+$ and that $\Gamma \cong \Z_{p+q}$ acts linearly via $\varphi^+$ or $\varphi^-$. Let $S^1_g$ be a $\Gamma$-invariant fiber of $\tilde{\mathcal{F}}_g$. By Lemma \ref{lem:exceptional_geo} the fibers $S^1_g$ and $\tilde{i}(S^1_g)$ are disjoint. Recall that the fibers of $\tilde{\mathcal{F}}_g$ come along with a natural orientation and that the map $\tilde{i}: S^1_g \To \tilde{i}(S^1_g)$ reverses these orientations. Let $\gamma$ be a generator of $\Gamma$ that rotates $S^1_g$ about a minimal angle in positive direction with respect to the natural orientation (cf. Section \ref{sub:linear} and Figure \ref{fig:rotations} and note that $\mathrm{ord}(\gamma)\geq 3$). Since $\Gamma$ and $\tilde{i}$ commute by Lemma \ref{lem:commute}, it follows that $\gamma$ rotates $S^1_g$ and $\tilde{i}(S^1_g)$ in different directions as depicted in Figure \ref{fig:rotations} in the case $p+q=5$. Now the fact that a generator of the $\Z_{p+q}$-action on $(S^3,\mathcal{H})$ via $\varphi^+$ rotates all fibers in the same direction implies that $\mathcal{F}_g=\mathcal{F}^-$ in view of Lemma \ref{lem:lens}. In particular, we have $\Orb_g=S^2((p+q)/\kappa,(p+q)/\kappa)$ with $\kappa$ being $1$ or $2$ depending on whether $p+q$ is odd or even by that lemma. Moreover, by Lemma \ref{lem:fix_point_i} it follows that $\Orb_g/i$ is either $\RP^2((p+q)/\kappa)$ or $D^2((p+q)/\kappa;)$ depending on whether $pq$ is odd or even. In particular, the geodesic periods are $(\overline{(p+q)/2})$ or $(1_{\infty},\overline{(p+q)/\kappa})$ depending on whether $pq$ is odd or even as discussed in Section \ref{sub:orb_of_geo}.
\begin{figure}
	\centering
		\includegraphics[width=0.5\textwidth]{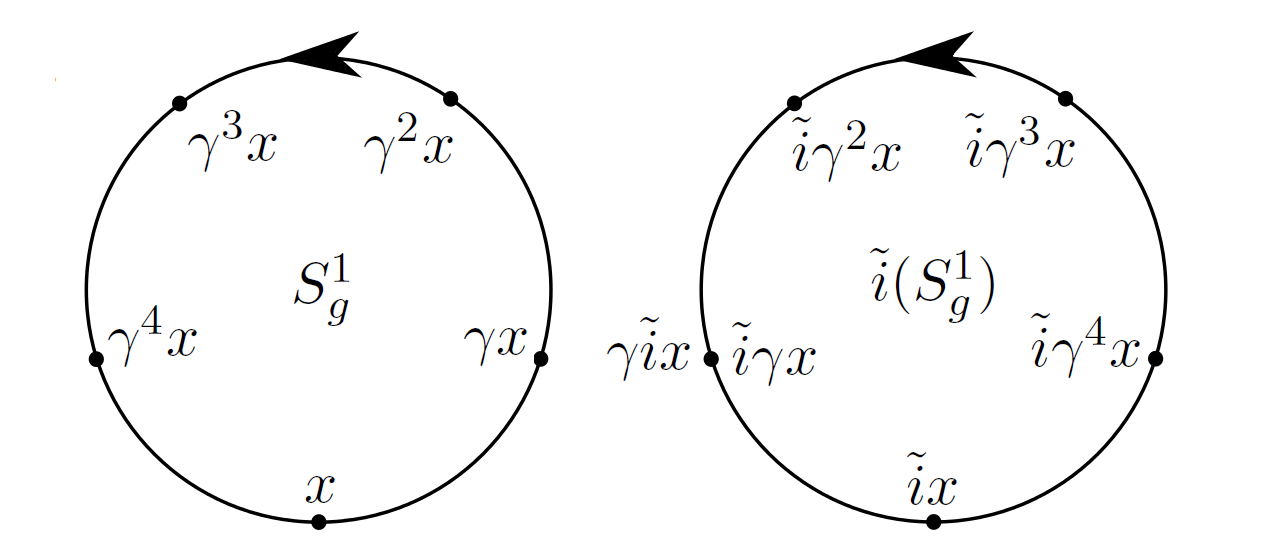}
	\caption{Illustration of an argument in Section \ref{sub:spindle}, f) in the case $p+q=5$. Note that $\gamma$ and $\tilde i$ commute by Lemma \ref{lem:commute}.}
	\label{fig:rotations}
\end{figure}
\begin{rem}\label{rem:spherical_case}Observe that in the case $p=q$ we have $\Orb \cong S^2/\mathrm{C}_p \cong \Orb_g$ and $\Orb_g/i \cong S^2/\left\langle \mathrm{C}_p,-1\right\rangle$, where $\mathrm{C}_p<\SOr(2) \subset \SOr(3)$ is a cyclic group of order $p$.
\end{rem}

\subsection{The real projective plane}\label{sub:pro_plane}

In this section we apply the above analysis to the real projective plane $\Orb=\RP^2$ endowed with a Besse metric. The unit tangent bundle $M=T^1 \RP^2$ of $\RP^2$ is homeomorphic to the lens space $L(4,1)$ \cite{MR1907070}. The fiberings $\mathcal{F}_g$ and $\mathcal{F}_t$ on $M$ are defined as in Section \ref{sub:spindle}, step b). The fibering $\mathcal{F}_g$ lifts to a fibering on the universal covering $\tilde M = S^3$ as in Section \ref{sub:spindle} by Remark \ref{rem:lifting}. The main difference to that section is that in the present case the fibering $\mathcal{F}_t$ is not of type $o_1$ since the quotient $M/\mathcal{F}_t = \RP^2$ is not orientable. The fibers of $\mathcal{F}_t$ cannot be oriented in a continuous way. Still, $\mathcal{F}_t$ lifts to the natural fibering of type $o_1$ on $T^1 S^2$ defined by $T^1 S^2 \To S^2$ and thus to a fibering $\tilde{ \mathcal{F}}_t$ on $\tilde M$ by Remark \ref{rem:lifting}. As in Section \ref{sub:spindle} Remark \ref{rem:lifting} also shows that the group $\Gamma = \mathrm{Deck}(\tilde{M} \To M)\cong \Z_4$ preserves the orientations of the fibers of $\tilde{\mathcal{F}}_g$. However, this argument does not apply to $\tilde{ \mathcal{F}}_t$ since $\mathcal{F}_t$ is not of type $o_1$. In fact, the deck transformation of the covering $T^1 S^2 \To M$ reverses the orientations of the fibers of $T^1 S^2 \To S^2$ and thus a generator of $\Gamma$ reverses the orientations of the fibers of $\tilde{ \mathcal{F}}_t$.

We choose a lift $\tilde{i}: \tilde M \To \tilde M$ of the involution $i: M \To M$, defined in Section \ref{sub:orb_of_geo}, that covers the natural lift $i: T^1 S^2 \To T^1 S^2$. As in Lemma \ref{lem:orientation_i_tilde} the map $\tilde{i}$ preserves the orientations of the fibers of $\tilde{ \mathcal{F}}_t$ while it reverses the orientation of the fibers of $\tilde{\mathcal{F}}_g$. The fact that in this case, compared to the situation in Section \ref{sub:spindle}, $\Gamma$ does not preserve the orientations of the fibers of $\tilde{ \mathcal{F}}_t$ results in the following statement converse to Lemma \ref{lem:commute}.

\begin{lem} \label{lem:no_commute} The group $\Gamma$ is normalized by $\tilde i$ and for a generator $\gamma$ of $\Gamma$ we have $\gamma \tilde i= \tilde i \gamma^3$.
\end{lem} 
\begin{figure}
	\centering
		\includegraphics[width=\textwidth]{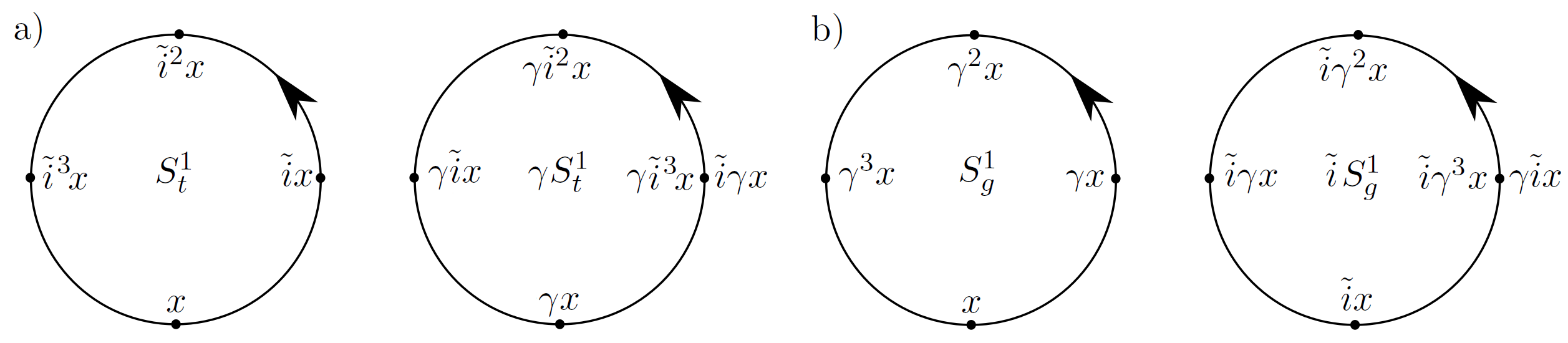}
	\caption{a) Illustration of an argument in Lemma \ref{lem:no_commute}. b) Illustration of an analogous argument. Note that $\tilde i \Gamma x= \Gamma \tilde i x$ since $\tilde i$ normalizes $\Gamma$.}
	\label{fig:rotations2}
\end{figure}
\begin{proof} As a lift of an involution of $M$ the map $\tilde i$ normalizes the group $\Gamma=\mathrm{Deck}(\tilde M\To M)$. Since $S^2$ is simply connected as an orbifold, the preimages of the fibers of $T^1S^2 \To S^2$ under the two-fold covering $\tilde{M}=S^3 \To T^1S^2$ are connected and $\Gamma_0 = \mathrm{Deck}(\tilde{M} \To T^1S^2)\cong \Z_2$ leaves the fibers of  $\tilde{ \mathcal{F}}_t$ invariant. Therefore also the lift $\tilde{i}: \tilde{M} \To \tilde{M}$ leaves the fibers of $\tilde{ \mathcal{F}}_t$ invariant (cf. proof of Lemma \ref{lem:commute}). Moreover, we have $\tilde{i}^2 \in \Gamma_0$ and we know that $\tilde{i}$ and $\Gamma_0$ commute by Lemma \ref{lem:commute}. Since both $\Gamma_0$ and $\tilde i$ leave some fiber $S_x^1$ of $\tilde{ \mathcal{F}}_t$ and its orientation invariant, their restrictions to $S_x^1$ generate a cyclic group (cf. Section \ref{sub:linear}). In a cyclic group at most a single nontrivial element squares to the identity, in this case the generator of $\Gamma_0$. Therefore, the fact that $\tilde{i}$ is not contained in $\Gamma_0$ implies that $\tilde{i}$ has order $4$ and that $\tilde i^2 = \gamma^2$ for some generator $\gamma$ of $\Gamma$. This shows that the restriction of $\tilde i$ to each fiber of $\tilde{\mathcal{F}}_t$ has order $4$, since $\gamma^2$ is a deck transformation distinct from the identity. Since $\tilde i$ is not contained in $\Gamma$ and normalizes $\Gamma$, the lift $\tilde i$ and $\Gamma$ generate a group $\tilde \Gamma$ of order $8$. In particular, the cyclic group $I$ of order $4$ generated by $\tilde i$ is normalized by $\Gamma$ (as an index $2$ subgroup of $\tilde \Gamma$). Let $S^1_{t}$, $\gamma S^1_{t}$ be a pair of fibers of $\tilde{\mathcal{F}}_t$ invariant under $\Gamma$ and pick a point $x \in S^1_{t}$. By Lemma \ref{lem:rot_same_direction} the map $\tilde i$ rotates $S^1_{t}$ and $\gamma S^1_{t}$ in the same direction. Recall that the map $\gamma : S^1_{t} \To \gamma S^1_{t}$ reverses orientations. We can assume that the actions of $\tilde i$ on $S^1_{t}$ and $\gamma S^1_{t}$ are linear (cf. Section \ref{sub:linear}). Because of $\gamma I x= I \gamma x$ it follows that the situation is as depicted in Figure \ref{fig:rotations2}, a). In particular, $\gamma \tilde i x$ and $\tilde i \gamma x$ differ and hence $\tilde i$ and $\Gamma$ do not commute, i.e. the generator $\tilde i^{-1} \gamma \tilde i$ of $\Gamma$ is distinct from $\gamma$. The only other generator of $\Gamma\cong \Z_4$ is $\gamma^3$ and so the second claim follows, too. 
\end{proof}
We claim that $\Orb_g=S^2$ and thus that all geodesics on $\RP^2$ have the same period. This was already shown in \cite{MR2481742} in a different way. The claim follows analogously as in Section \ref{sub:spindle} from Lemma \ref{lem:no_commute}: By the same reason as in Section \ref{sub:spindle} the $S^1$-action on $M$ that is induced by the geodesic flow and that defines $ \mathcal{F}_g$ satisfies the assumption of Lemma \ref{lem:lens}. More precisely, the arguments in Section \ref{sub:spindle} c) show that $S^3/\tilde{\mathcal{F}}_g\cong S^2(p_g,q_g)$ 
for coprime $p_g$, $q_g$, that $\Gamma$ fixes the singularities of $S^3/\tilde{\mathcal{F}}_g$ and preserves its orientation and thus that $M/\mathcal{F}_g \cong S^2(kp_g,kq_g)$ for some $k||\Gamma|$. Moreover, the fact that $i$ acts freely on $\Orb_g$ by Lemma \ref{lem:fix_point_i} implies $p_g=q_g=1$ as in Section \ref{sub:spindle}, d). Hence, by Lemma \ref{lem:lens} $(M,\mathcal{F}_g)$ is either smoothly conjugated to $(L(4,1),\mathcal{F}^+)$ or to $(L(4,1),\mathcal{F}^-)$, we can assume that $\tilde{\mathcal{F}}_g$ is the Hopf fibration $\mathcal{H}$ on $S^3$ defined by $\varphi^+$, and that $\Gamma \cong \Z_{4}$ acts linearly via $\varphi^+$ or $\varphi^-$, respectively. Now let $S^1_g$ be a $\Gamma$-invariant fiber of $\tilde{\mathcal{F}}_g$ and let $\gamma$ be a generator of $\Gamma$ that rotates $S^1_g$ about a minimal angle in positive direction with respect to the natural orientation (cf. Figure \ref{fig:rotations2}, b) and Section \ref{sub:linear}). Since $\Gamma$ acts freely on $S^1_g$, the $\Gamma$-orbit of a point $x \in S^1_g$ has order four as depicted in Figure \ref{fig:rotations2}, b). Since $\RP^2$ has no orbifold singularities, the fibers $S^1_g$ and $\tilde{i}S^1_g$ are distinct by Lemma \ref{lem:fix_point_i}. Now the fact that $\tilde i$ reverses the orientation of the $\tilde{\mathcal{F}}_g$-fibers (see Lemma \ref{lem:orientation_i_tilde}) and Lemma \ref{lem:no_commute} imply that the $\tilde i$-images of $S^1_g$ and $\Gamma x$ look as depicted in Figure \ref{fig:rotations2}, b). In particular, we see that $\gamma$ rotates $S^1_g$ and $\tilde{i}S^1_g$ in the same direction. Observe that in case of a $\Z_{4}$-action on $\mathcal{H}$ via $\varphi^-$ the only two $\Gamma$-invariant fibers are rotated in different directions (cf. Section \ref{sub:linear}). We conclude that $\mathcal{F}_g=\mathcal{F}^+$, $\Orb_g \cong S^2$ and $\Orb_g /i\cong \RP^2$ as orbifolds by Lemma \ref{lem:lens} and Lemma \ref{lem:fix_point_i}. In particular, all geodesics have the same length.

In Section \ref{sub:rigidity}, we show how this result can be used to deduce rigidity on the real projective plane. This implies that every Besse orbifold covered by the real projective plane has constant curvature.

\subsection{Non-orientable half-spindle orbifolds} \label{sub:nonorbad}
In this section we assume $\Orb$ to be a Besse $2$-orbifold of type $D^2(;p,q)$. Let $\hat \Orb$ be its orientable double cover of type $S^2(p,q)$ and let $s$ be the deck transformation of the covering $\hat \Orb \To \Orb$. From Section \ref{sub:spindle} we know that $\hat \Orb_g \cong S^2((p+q)/\kappa,(p+q)/\kappa))$ where $\kappa$ is $1$ or $2$ depending on whether $p+q$ is odd or even. Since the action of $s$ on $\hat \Orb_g$ preserves the orientation and can be linearized (cf. Section \ref{sub:linear}), it is either trivial or fixes precisely two points (cf. Section \ref{sub:orb_of_geo}). We claim that the latter is always the case. To see this it suffices to find a geodesic that is not invariant under $s$. If $pq$ is odd, then there are no self-inverse geodesics $\hat \Orb$ and so any geodesic that hits $\mathrm{Fix}(s)\subset \hat \Orb$ perpendicularly is not invariant under $s$. For even $pq$ the geodesics in $\mathrm{Fix}(s)\subset \hat \Orb$ hit a singular point of even order and are thus regular by Lemma \ref{lem:exceptional_geo}. Choose a regular geodesic $c$ that lies in $\mathrm{Fix}(s) \subset \hat{\Orb}$ and starts at the singularity $x \in \hat{\Orb}$ of order $p$ in a direction $v\in T^1_x \hat\Orb$. By continuity there is a neighborhood $U$ of $v$ in $T^1_x \hat\Orb$ such that any geodesic that starts at $x$ in a direction $\tilde v \in U$ intersects $U$ at most once. We can assume that $U$ is invariant under $s$. It follows that any geodesic starting at $x$ in a direction $\tilde v \in U$ which is different from $v$ is not invariant under $s$. Hence, in any case $s$ fixes precisely two points on $\hat{\Orb}_g$. Among the fixed points on $\hat{\Orb}_g$ are the geodesics contained in $\mathrm{Fix}(s) \subset \hat{\Orb}$. As already mentioned, for even $pq$ these geodesics are regular by Lemma \ref{lem:exceptional_geo}. In the presence of the symmetry $s$ this property holds regardless of the parity of $pq$:

\begin{lem}\label{lem:reg_in_fixset} The geodesics contained in $\mathrm{Fix}(s) \subset \hat{\Orb}$ are regular.

\end{lem}
\begin{proof} Suppose a geodesic $c$ in $\mathrm{Fix}(s)$ is exceptional. Then we must have $p+q >2$, as there would not exist exceptional geodesics otherwise by Theorem B, $(ii)$, which has been proven in Section \ref{sub:spindle}. In particular, there is some singular point $x$ in $\mathrm{Fix}(s)$ hit by $c$. By Lemma \ref{lem:exceptional_geo} we can assume that $pq$ is odd. By our assumptions the fiber $S^1_c$ of $\mathcal{F}_g$ corresponding to $c$ and the fiber $S^1_x=T_x^1 \hat \Orb$ of $\mathcal{F}_t$ intersect in a single point. The fiber $S_c^1$ is pointwise fixed by $s$ while the fiber $S^1_x$ is reflected about two points. Since $S^1_c$ and $S^1_x$ are singular by assumption, their preimages under the covering $S^3 \To T^1 \hat \Orb$ are connected fibers $\tilde S^1_c$ of $\tilde{\mathcal{F}}_g$ and $\tilde S^1_x$ of $\tilde{\mathcal{F}}_t$ invariant under $\Gamma=\mathrm{Deck}(S^3 \To T^1 \hat \Orb)$ (cf. proof of Lemma \ref{lem:commute}). In particular, both $\tilde S^1_c$ and $\tilde S^1_x$ are invariant under each lift $\tilde s$ of $s$ to $S^3$. Since $s$ is the identity on $S^1_c$ we can choose a lift $\tilde s$ which is the identity on $\tilde S^1_c$. Since $\tilde s$ reverses the orientation of $\tilde S^1_x$, it acts as a reflection (see Section \ref{sub:linear}) that fixes precisely two points on $\tilde S^1_x$. Both together implies that $|\tilde S^1_c \cap \tilde S^1_x|\leq 2$. Because of $|\tilde S^1_c \cap \tilde S^1_x|=|\Gamma||S^1_c \cap S^1_x|=|\Gamma|=p+q$ (see Lemma \ref{lem:gluing} for the third equality), we conclude that $p+q \leq 2$, a contradiction. Hence the claim follows.
\end{proof}

Now we apply Lemma \ref{lem:pro_geo_per} in order to determine the geodesic periods. Recall that $s$ preserves the orbifold structure of $\hat{\Orb}_g$ and from Section \ref{sub:spindle}, f), the geodesic periods of $S^2(p,q)$.

If $p+q$ is even, there are precisely two geodesics contained in $\mathrm{Fix}(s)$. These geodesics are regular by Lemma \ref{lem:reg_in_fixset} and pointwise fixed by $s$. It follows that $\Orb_g=\hat \Orb_g/s\cong S^2(2,2,(p+q)/2)$ and that the geodesic periods of $\Orb$ are $(1_{\infty},(p+q)/2)$. Since $\Orb$ has a topological boundary, by Lemma \ref{lem:fix_point_i} we either have $\Orb_g/i\cong D^2(;2,2,(p+q)/2)$ or $\Orb_g/i\cong D^2(2;(p+q)/2)$. These cases occur depending on whether $pq$ is even or odd since by Lemma \ref{lem:fix_point_i} these are the conditions for the geodesics in $\mathrm{Fix}(s)$, which correspond to the singularities of order $2$ on $\Orb_g$, to be invariant under $i$ or not.

If $p+q$ is odd there is only one geodesic contained in $\mathrm{Fix}(s)$. This geodesic is regular by Lemma \ref{lem:reg_in_fixset} and pointwise fixed by $s$. The other $s$-invariant geodesic must also be regular, because otherwise both singular geodesics would have to be invariant, contradicting the fact that $s$ has only two fixed points on $\hat \Orb_g$. Moreover, this geodesic is not pointwise fixed and thus projects to a geodesic of half the period. Hence, in this case we have $\Orb_g\cong S^2(2,2,(p+q))$ and the geodesic periods are $(1_{\infty},2,(p+q))$. Since $pq$ is even it follows as above that $\Orb_g/i\cong D^2(;2,2,(p+q)/2)$.

\begin{rem}\label{rem:spherical_case_half}Observe that in the case $p=q$ we have $\Orb \cong S^2/\mathrm{D}_p \cong \Orb_g$ and $\Orb_g/i \cong S^2/\left\langle \mathrm{D}_p,-1\right\rangle$, where $\mathrm{D}_p< \SOr(3)$ is a dihedral group of order $2p$.
\end{rem}

\subsection{Orientable spherical orbifolds}
\label{sub:orientable}

The orientable spherical orbifolds are the quotients of $S^2$ by the finite subgroups of $\SOr(3)$. These are cyclic groups $\mathrm{C}_n\cong \Z_n$ of order $n$, dihedral groups $\mathrm{D}_n\cong \mathfrak{D}_n$ of order $2n$ and tetrahedral, octahedral and icosahedral groups $\mathrm{T}$, $\mathrm{O}$ and $\mathrm{I}$ isomorphic to the alternating group $\mathfrak{A}_4$, the symmetric group $\mathfrak{S}_4$ and the alternating group $\mathfrak{A}_5$, respectively. Let us suppose that $\Orb$ is a quotient of $S^2$ by one of these groups, call it $G$, endowed with a Besse metric. In other words, we have a $G$-invariant Besse metric on $S^2$ and we denote this Besse manifold by $\tilde \Orb$. From Theorem B, $(ii)$ in the case $p=q=1$, which has been proven in Section \ref{sub:spindle}, we know that $(\tilde \Orb)_g \cong S^2$ as orbifolds. Note that $(\tilde \Orb)_g$ is the simply connected orbifold cover of $\Orb_g$. Indeed, we have an ordinary covering $T^1 \tilde \Orb \To T^1 \Orb$ of Seifert fibered manifolds, with the fiberings being induced by the geodesic flow lines, and deck transformation group $G$. After collapsing the fibers this covering induces a covering of orbifolds $(\tilde \Orb)_g \To \Orb_g$ with deck transformation group $G$, i.e. we have $\Orb_g\cong (\tilde \Orb)_g/G$ for the induced action of $G$ on $(\tilde \Orb)_g$. Therefore, we simply write $\tilde \Orb_g:=(\tilde \Orb)_g$.

Theorem B, $(ii)$ and $(iii)$ in the case $p=q$, moreover show that every rotation in $G$ acts non-trivially as a rotation on $\tilde \Orb_g$. Since every nontrivial element in $G$ is a rotation, it follows that $G$ acts effectively on $\tilde \Orb_g$ and preserves the orientation. By Lemma \ref{lem:commute} the actions of $G$ and $i$ on $\tilde \Orb_g$ commute. Therefore these actions can be linearized simultaneously (cf. Section \ref{sub:linear}). Since the faithful representation of $G$, as an abstract group, in $\SOr(3)$ is unique, it follows that $\Orb_g=\tilde \Orb_g /G \cong \Orb$ as orbifolds. Since $i$ acts freely on $\tilde \Orb_g$ as an involution, its linearization in $\SOr(3)$ has to be minus the identity. Therefore, we have $\Orb_g/i\cong S^2/\Gs$ with $\Gs=\left\langle G,-1 \right\rangle$. Note that these results are compatible with Section \ref{sub:spindle} and Section \ref{sub:nonorbad}, see Remark \ref{rem:spherical_case} and Remark \ref{rem:spherical_case_half}. This finishes the proof of Theorem B in the case of orientable spherical orbifolds, and hence also the proof of Theorem A in this case as discussed in the last paragraph of Section \ref{sub:orb_of_geo}.

To compute the geodesic periods recall from Section \ref{sub:orb_of_geo} how they are determined by the covering $\Orb_g \To \Or_g/i$ in the case of a $2$-orbifold $\Orb$ with only isolated singularities. The only remaining cases in which this determination is not immediately clear from the combinatorial restriction that $i$ preserves the orders of the local groups of $\Orb_g$ correspond to the groups $\mathrm{D}_n$ and $\mathrm{T}$. In these cases we give a more geometric description which will be needed later.
\newline
\newline
\noindent 2., 3. If $G =\mathrm{D}_n$, $n\geq 2$, then $\Orb,\Orb_g\cong S^2(2,2,n)$ have double covers $\hat{\Orb},\hat{\Orb}_g \cong S^2(n,n)$. For even $n$ there exist two geodesics on $\hat{\Orb}$ connecting the two singular points on $\hat{\Orb}$ of order $n$ and passing a branch point of the covering $\hat{\Orb} \To \Orb$. In fact, we can choose minimizing segments between one of these singularities and the two branch points. Since the covering is two-fold these segments extend to trajectories of the desired geodesics. These two geodesics are regular by Lemma \ref{lem:exceptional_geo}. Moreover, by construction they project to exceptional geodesics on $\Orb$ of half the (regular) period that are fixed by $i$. Hence, we have $\Orb_g/i \cong D^2(;2,2,n)$ and the geodesic periods are $(1_{\infty},2,2,n)$.

For odd $n$ geodesics on $\hat{\Orb}$ passing a branch point of the covering $\hat{\Orb} \To \Orb$ project to geodesics on $\Orb$ of the same period. In particular, the geodesics of order $2$ on $\Orb$ do not hit singularities of even order. Therefore, by Lemma \ref{lem:fix_point_i} we have $\Orb_g/i \cong D^2(2;n)$ and the geodesic periods are $(1_{\infty},\overline{2},n)$. We record the following lemma.
\begin{lem}\label{lem:hit_branch}
An exceptional geodesic on $\hat{\Orb}\cong S^2(n,n)$ hits one of the branch points of $\hat{\Orb} \To \Orb\cong S^2(2,2,n)$, regardless of the parity of $n$.
\end{lem}
\begin{proof}
By our discussion above and Lemma \ref{lem:fix_point_i} a geodesic of order $n$ on $\Orb$ has to hit a singular point of even order. By Lemma \ref{lem:exceptional_geo} it has to be a branch point of the covering $\hat{\Orb} \To \Orb$.
\end{proof}
\noindent 4. If $G =\mathrm{T}$, then $\Orb,\Orb_g\cong S^2(2,3,3)$. A geodesic $c$ of order $3$ on $\Orb$ (corresponding to a singular point of order $3$ on $\Orb_g$)  is three-foldly covered by a geodesic on $\tilde \Orb \cong S^2$ that is invariant under a subgroup of $\mathrm{T}$ of order $3$. If $c$ hit a singularity of order $2$ on $\Orb$, the trajectory of its lift would be invariant under a subgroup of $\mathrm{T}$ of order $\geq 6$ and thus invariant under $\mathrm{T}$, since there are no subgroups of order $6$ in $\mathrm{T}$. However, then the order of $c$ would have to be at least $6$ which is a contradiction. Hence, $c$ does not hit a singularity of even order and thus, using Lemma \ref{lem:fix_point_i}, we have $\Orb_g/i \cong D^2(3;2)$ and the geodesic periods are $(1_{\infty},2,\overline{3})$.

\subsection{Non-orientable spherical orbifolds}\label{sub:non-orient_spher}
Suppose that $\Orb$ is a quotient of $S^2$ by a finite subgroup $G$ of $\Or(3)$ that does not preserve the orientation. Suppose further that $\Orb$ is endowed with a Besse metric. In other words, we have a $G$-invariant Besse metric on $S^2$ and we denote this Besse manifold by $\tilde \Orb$. Let $\Gp$ be the orientation preserving subgroup of $G$ and set $\hat \Orb = \tilde \Orb / \Gp$. By Section \ref{sub:orientable} we know that $\tilde \Orb_g:=(\tilde \Orb)_g\cong S^2$ is the simply connected covering orbifold of $\hat \Orb_g$ and that $\hat \Orb_g \cong \tilde \Orb_g / \Gp$. Recall from Section \ref{sub:orb_of_geo} that $\Orb_g$ was defined as $\Orb_g=\hat \Orb_g /(G/\Gp)$. Hence, $\tilde \Orb_g$ is also the simply connected covering orbifold of $\Orb_g$ and we have $\Orb_g \cong \tilde \Orb_g / G$ with respect to the induced action of $G$ on $\tilde \Orb_g$.

In the case $-1 \in G$ it follows from Section \ref{sub:pro_plane} that $\Orb$ has constant curvature and that $\Orb_g \cong \hat \Orb_g $. In the case $-1 \notin G$ we claim that $G$ acts effectively on $\tilde \Orb_g$. Indeed, in this case the $\mathrm{ord}(g)/2$-th power of an element $g \in G$ with $\mathrm{det}(g)=-1$ is a reflection and this reflection does not leave invariant the geodesics that hit its fixed point set perpendicularly. Hence, in this case we have $\Orb_g\cong S^2/\Gt$ for a finite subgroup $\Gt$ of $\SOr(3)$ that is abstractly isomorphic to $G$. The fact that a finite subgroup of $\SOr(3)$ is determined by its abstract isomorphism type implies that $\Gt= \{\mathrm{det}(g)g| g \in G\}$. Hence, in any case we have $\Orb_g\cong S^2/\Gt$, and $\Orb_g/i\cong S^2 / \Gs$ with $\Gs=\left\langle G, -1\right\rangle$ since $i$ acts on $\tilde \Orb$ as an inversion by Lemma \ref{lem:fix_point_i}. In particular, we have $\Orb_g/i \cong \Orb$ as orbifolds if $-1 \in G$.

It remains to prove rigidity of the geodesic periods in the respective cases and to compute them. Let $s$ be the deck transformation of the covering $\hat \Orb \To \Orb$. As seen in Section \ref{sub:linear} the involution $s$ either acts trivially on $\hat \Orb_g$ or fixes precisely two points. As seen above, it acts trivially if and only if $-1 \in G$. In order to determine the geodesic periods of $\Orb$, we have to identify the $s$-invariant geodesics on $\Orb$ and decide whether they are pointwise fixed or not. Then Lemma \ref{lem:pro_geo_per} tells us how the geodesic periods of the quotient look like. We go through the cases with $-1 \notin G$ listed under number 8.-12. in Table \ref{tab:gls} and use the Sch\"{o}nflies notation to specify the group $G$ that defines $\Orb\cong S^2/G$ (cf. Table \ref{tab:gls} in the appendix). Note that in all cases $\Orb_g/i$ is uniquely determined as an orbifold by its double covers $\Orb_g$ and $\Orb$. The case 7. with $\Gp=\mathrm{C}_n$ has already been treated in Section \ref{sub:nonorbad}. The cases with $-1 \in G$, listed under number 13.-19. in Table \ref{tab:gls}, in which a Besse metric has constant curvature, can be treated analogously. We only give details on case 14. as an example. We frequently use Lemma \ref{lem:fix_point_i} and the identification of geodesics as elements of the orbifold of geodesics without explicitly mentioning it.
\newline
\newline
\noindent 8. In the case of $\Orb\cong \RP^2(2n)\cong S^2/\mathrm{S}_{4n}$ we have $\hat\Orb\cong S^2(2n,2n) \cong \hat \Orb_g$ where $\mathrm{S}_{4n}\cong \Z_{4n}\cong \mathrm{C}_{4n}$. Hence, $\Orb_g\cong S^2/\mathrm{C}_{4n} \cong S^2(4n,4n)$ and $\Orb_g/i\cong D^2(4n;)$. Since $\Orb$ has only isolated singularities we see that the geodesic periods are $(1_{\infty},\overline{4n})$ (see Section \ref{sub:orb_of_geo}).
\newline
\newline
\noindent 9. In the case of $\Orb\cong D^2(2n+1;)\cong S^2/\mathrm{C}_{2n+1\mathrm{h}}$ we have $\hat\Orb\cong S^2(2n+1,2n+1)\cong \hat\Orb_g$ and $\hat\Orb_g/i\cong \mathbb{RP}^2(2n+1)$ where $\mathrm{C}_{2n+1\mathrm{h}}\cong \Z_{2n+1}\times \Z_2 \cong \Z_{4n+2} \cong \mathrm{C}_{4n+2}$. Hence, $\Orb_g\cong S^2/\mathrm{C}_{4n+2} \cong S^2(4n+2,4n+2)$ and $\Orb_g/i\cong D^2(4n+2;)$. In particular, the exceptional geodesics on $\hat \Orb$ are invariant under $s$. Since the two geodesics in the fixed point set of $s$ are also invariant, we see that these must be the exceptional geodesics. It follows that the geodesic periods are $(1_{\infty},\overline{2n+1})$.  
\newline
\newline
\noindent 10. In the case of $\Orb\cong D^2(2;2n)\cong S^2/\mathrm{D}_{2n\mathrm{d}}$ we have $\hat\Orb\cong S^2(2,2,2n) \cong \hat\Orb_g$ and $\hat\Orb_g/i\cong D^2(;2,2,2n)$ where $\mathrm{D}_{2n\mathrm{d}}\cong \mathfrak{D}_{4n} \cong \mathrm{D}_{4n}$. Hence, $\Orb_g \cong S^2/\mathrm{D}_{4n}\cong S^2(2,2,4n)$ and $\Orb_g/i \cong D^2(;2,2,4n)$. In particular, invariant under $s$ are an exceptional geodesic of order $2n$ and a regular geodesic. By Lemma \ref{lem:hit_branch} this exceptional geodesic of order $2n$ hits a singularity of order $2$ that is not fixed by $s$. Therefore, the invariant exceptional geodesic is not pointwise fixed, only the regular geodesic in $\mathrm{Fix}(s)\subset \hat\Orb$ is so. We conclude that the geodesic periods are $(1_{\infty},2,4n)$.
\newline
\newline
\noindent 11. In the case of $\Orb\cong D^2(;2,2,2n+1)\cong S^2/\mathrm{D}_{2n+1\mathrm{h}}$ we have $\hat\Orb\cong S^2(2,2,2n+1)\cong \hat\Orb_g$ and $\hat\Orb_g/i\cong D^2(2;2n+1)$ where $\mathrm{D}_{2n+1\mathrm{h}}\cong \mathfrak{D}_{2n+1}\times \Z_2 \cong \mathrm{D}_{4n+2}$. Hence, $\Orb_g\cong S^2/\mathrm{D}_{4n+2} \cong S^2(2,2,4n+2)$ and $\Orb_g/i\cong  D^2(;2,2,4n+2)$. In particular, invariant under $s$ are the exceptional geodesic of order $2n+1$ and a regular geodesic. These geodesics are contained in $\mathrm{Fix}(s)\subset \hat\Orb$, since there are two geodesics in $\mathrm{Fix}(s)$ and since $s$ only leaves two geodesics invariant. We conclude that the geodesic periods are $(1_{\infty},2,2n+1)$.
\newline
\newline
\noindent 12. In the case of $\Orb \cong D^2(;2,3,3)\cong S^2/\mathrm{T}_{\mathrm{d}}$ we have $\hat\Orb\cong S^2(2,3,3) \cong \hat\Orb_g$ and $\hat\Orb_g/i\cong D^2(3;2)$ where $\mathrm{T}_{\mathrm{d}} \cong \mathfrak{S}_4 \cong \mathrm{O}$. Hence, $\Orb_g\cong S^2/\mathrm{O} \cong S^2(2,3,4)$ and $\Orb_g/i\cong D^2(;2,3,4)$. In particular, invariant under $s$ is the exceptional geodesic of order $2$ and a regular geodesic. We claim that the former is not pointwise fixed by $s$. To see this we consider a three-fold orbifold covering $S^2(2,2,2)\cong \Orb' \To \hat\Orb\cong S^2(2,3,3)$. In Section \ref{sub:orientable}, 2., we have seen that the three exceptional geodesics of order $2$ on $ \Orb'$ are defined by minimizing segments $c_1$, $c_2$, $c_3$ connecting the three singularities of order $2$ pairwise. Due to the minimizing property, their trajectories have a trivial intersection. Indeed, if $c_1$ and $c_2$ intersected away from the singularities, the intersection point would divide $c_1$ and $c_2$ in equal proportions by their minimizing property. But this would contradict the uniqueness of minimizing segments between singularities of order $2$ (any such segment gives rise to an exceptional closed geodesic oscillating between the singularities of even order as discussed in Section \ref{sub:orientable}, 2., but in total there are only three exceptional geodesics on $\Orb'$). Hence, the three exceptional geodesics on $\Orb'$ do not pass through the branch points of the covering $\Orb' \To \hat\Orb$. In other words, the exceptional geodesic of order $2$ on $\hat\Orb$ does not hit a singularity of order $3$ and is thus not contained in $\mathrm{Fix}(s)$. The other, regular invariant geodesic is the unique geodesic in $\mathrm{Fix}(s)\subset \hat\Orb$ which is pointwise fixed. We conclude that the geodesic periods are $(1_{\infty},3,4)$.
\newline
\newline
\noindent 14. In the case of $\Orb \cong D^2(2n;)\cong S^2/\mathrm{C}_{2n\mathrm{h}}$ we have $\hat\Orb\cong S^2(2n,2n) \cong \hat\Orb_g\cong\Orb_g$ and $\Orb_g/i\cong \Orb$. In particular, all geodesics on $\Orb$ are invariant under $s$, but only the two exceptional geodesics in $\mathrm{Fix}(s)$ are pointwise fixed by $s$ (recall that $\Orb$ has constant curvature in this case). Therefore, the exceptional geodesics on $\hat \Orb$ project to geodesics of the same period, while all other geodesics project to geodesics of half the period. We conclude that the geodesic periods are $(1_{\infty},\overline{n})$.

\section{Conjugacy of induced contact structures} \label{sub:conjugacy_contact_structures}

Let $\Orb$ be a Besse $2$-orbifold with only isolated singularities. As in the manifold case, see \cite[Appendix B]{ABHS16} and \cite[Thm.~1.5.2]{MR2397738}, the unit tangent bundle $M=T^1 \Orb$ carries a natural contact $1$-form $\alpha$ whose Reeb flow is the geodesic flow on $M$. The form $\alpha$ is the restriction of the pullback of the canonical Liouville form on $T^* \Orb - \Orb$ via the isomorphism $T \Orb - \Orb \To T^* \Orb - \Orb$ induced by the metric (here we regard $\Orb \subset T^{(*)} \Orb$ as the zero-section). The aim of this section is to prove the following result.
\begin{thm}\label{thm:Reeb_conjugacy} Let $g_0$ and $g$ be Besse metrics with the same minimal period on a $2$-orbifold $\Orb$ with only isolated singularities. Let $\alpha_0$ and $\alpha$ be the corresponding contact forms on $T^1 (\Orb,g_0)$ and $T^1 (\Orb,g)$, respectively. Then there exists a diffeomorphism
\[
		\varphi : T^1 (\Orb,g_0) \To T^1 (\Orb,g)
\]
such that $\varphi^* \alpha = \alpha_0$.
\end{thm}
\begin{proof}
For the $2$-sphere this result is proven in \cite[Appendix B]{ABHS16}. We explain how the same methods can be used to prove the more general case. 

In case of the real projective plane and of spindle orbifolds we have seen in Lemma \ref{lem:lens} that the $S^1$-actions on $T^1 \Orb$ induced by the geodesic flows are conjugated by a diffeomorphism. Given Lemma \ref{lem:seifert_classif} the same statement is true for the remaining orbifolds in question, i.e. the spherical orbifolds $S^2(2,2,n)$, $S^2(2,3,3)$, $S^2(2,3,4)$ and $S^2(2,3,5)$, since their unit tangent bundles admit up to orientation only one Seifert fibering \cite[Thm.~10.2, p.~72]{MR0741334}. Hence, in each case there exists a diffeomorphism $\psi : T^1 (\Orb,g_0) \To T^1 (\Orb,g)$ that conjugates the $S^1$-actions induced by the geodesic flows. By assumption on the minimal periods the Reeb vector fields of $\alpha_0$ and $\alpha_1:= \psi^* \alpha$ coincide. Moreover, we claim that $\alpha_0$ and $\alpha_1$ define the same orientation, i.e. that $\psi$ preserves the orientations defined by $\alpha_0 \wedge d\alpha_0$ and $\alpha \wedge d\alpha$. It is sufficient to show that the pulled back contact forms $\tilde \alpha_0$ and $\tilde \alpha_1$ on the universal cover $S^3$ define the same orientation. The Reeb flows of $\tilde \alpha_0$ and $\tilde \alpha_1$ coincide and are defined by an $S^1$-action (cf. Remark \ref{rem:lifting}). From Sections \ref{sub:spindle} and \ref{sub:orientable} we know that $S^3/S^1 \cong S^2$ as orbifolds. In other words, the $S^1$-action defines a principal $S^1$-bundle with base $B=S^3/S^1\cong S^2$. By (the easy part of) a theorem by Boothby  and  Wang, see \cite{MR0112160} and \cite[Thm.~7.2.5]{MR2397738}, $\tilde \alpha_0$ and $\tilde \alpha_1$ are connection $1$-forms of this principal bundle, whose curvature forms $\omega_0$ and $\omega_1$ are area forms on $B$ satisfying $p^* \omega_0=d \tilde \alpha_0$ and $p^* \omega_1=d \tilde \alpha_1$, where $p: S^3 \To B$ is the natural projection. Moreover, $-[\omega_0 / 2\pi]=-[\omega_1 / 2\pi]$ is the Euler class of the principal $S^1$-bundle. In particular, $\omega_0$ and $\omega_1$ induce the same orientation on $B$ and hence $\tilde \alpha_0$ and $\tilde \alpha_1$ induce the same orientation on $S^3$. 

Now, as in \cite{ABHS16}, one can show that $\alpha_t=t \psi^*\alpha +(1-t) \alpha_0$ is a contact form for every $t \in [0,1]$ and apply Moser's argument to find a one-parameter family of diffeomorphisms $\phi_t: T^1 (\Orb,g_0) \To T^1 (\Orb,g_0)$, $t \in [0,1]$, such that $\phi_t^* \alpha_t = \alpha_0$ for every $t \in [0,1]$. In particular, 
\[
	\alpha_0=\phi_1^* \alpha_1 = \phi_1^* \psi^* \alpha
\]
and $\varphi=\psi \circ \phi_1$ is the desired diffeomorphism.
\end{proof}

\section{Appendix}

\subsection{Rigidity on the real projective plane} \label{sub:rigidity}

For a Riemannian metric $g$ on $\Orb=\RP^2$ we denote its corresponding area measure by $\nu_{g}$, its total area by $A_{g}$ and the length of a shortest non-contractible loop by $a_{g}$. The following inequality is due to Pu \cite{MR0048886}
\[
			A_g \geq \frac{2}{\pi} a_g^2.
\]
We recall its proof and show how it implies rigidity for Besse metrics on $\RP^2$. This argument was explained to us by A. Abbondandolo.

Suppose that $g$ is some Riemannian metric on $\RP^2$ and let $g_0$ be the standard Riemannian metric on $\RP^2$ of constant curvature $1$. The group $G=\SOr(3)$ acts on $\RP^2$ in its standard way. By the uniformization theorem there is some positive smooth function $\varphi$ on $\RP^2$ such that $g=\varphi \cdot g_0$. We endow $G$ with its Haar measure $\mu$ and define
\[
		\overline{\varphi}=\left(\int_G (g^* \varphi)^{1/2} d \mu \right)^2
\]
and $\overline{g}= \overline{\varphi} \cdot g_0$. By construction $\overline{g}$ is a $G$-invariant Riemannian metric on $M=\RP^2$ and hence has constant curvature. We claim that $A_g \geq A_{\overline{g}}$ and $a_{\overline{g}} \geq a_g$ . Indeed, we have 
\begin{equation*}
\begin{split}
A_{\overline{g}} &= \int_M \overline{\varphi} d\nu_{g_0}=\int_M  \left( \int_G (h^* \varphi)^{1/2} d \mu  \right)^2 d\nu_{g_0} \\ &\leq \int_M  \left( \int_G h^*\varphi d \mu  \right) d\nu_{g_0} =  \int_G  \left( \int_M h^*\varphi d\nu_{g_0} \right) d \mu = \int_G  A_g d \mu =A_g 
\end{split}
\end{equation*}
where we have applied the Cauchy-Schwarz inequality. Moreover, with a shortest non-contractible loop (and hence geodesic) $\gamma$ on $\RP^2$ with respect to $\overline{g}$ we have
\begin{equation*} \label{eq1}
\begin{split}
a_{\overline{g}}&= \int_0^1 \overline{\varphi}(\gamma(s))^{1/2} ||\dot{\gamma}(s)||_{g_0}ds=  \int_0^1 \left(\int_G ((h^* \varphi)(\gamma(s)))^{1/2} ||\dot{\gamma}(s)||_{g_0} d \mu \right) ds \\
& =\int_G \left(\int_0^1 ((h^* \varphi)(\gamma(s)))^{1/2} ||\dot{\gamma}(s)||_{g_0} ds \right) d \mu \geq \int_G a_g d \mu = a_g.
\end{split}
\end{equation*}
In particular, this proves Pu's inequality, since we have $A_{\overline{g}} = \frac{2}{\pi} a_{\overline{g}}^2$ for the metric of constant curvature $\overline{g}$ \cite{MR0048886}. Now suppose that $g$ is Besse. Theorem \ref{thm:Reeb_conjugacy} implies that the same equality also holds for the Besse metric $g$. In fact, after normalizing $g$ such that $a_g=\pi$ Theorem \ref{thm:Reeb_conjugacy} implies $\mathrm{vol}(T^1\RP^2, \alpha_0\wedge d \alpha_0)=\mathrm{vol}(T^1\RP^2, \alpha\wedge d \alpha)$ which in turn implies $A_g=2\pi$ by fiberwise integration with respect to $T^1\RP^2 \To \RP^2$. (Alternatively, this follows from a theorem of Weinstein: The two-fold covering $(S^2,\hat g)$ of $(\RP^2,g)$ has area $2A_g$ and the minimal geodesic period is $2a_g$ due to the fact that $\Orb_g\cong S^2$. Now the theorem by Weinstein says that for a Besse metric $\hat g$ on $S^2$ we have $\mathrm{area}(S^2,\hat g)=\frac{l^2}{\pi} $, where $l$ is the minimal geodesic period \cite{MR0390968}, cf. \cite[Prop.~2.24]{MR0496885}). It follows that $A_{\overline{g}}= A_g$, i.e. we have equality in the Cauchy-Schwarz inequality implying that $\varphi$ is constant. Hence, $g$ is proportional to $g_0$ and has constant curvature.

\begin{table}[p] \centering
\begin{tabular}{ |r|c| l | l| l | l |}
\hline                       
  & $(G<\Or(3))$ &$\Orb(\cong S^2/G)$ & $\Orb_g(\cong S^2/\Gt)$&$\Orb_g/i(\cong S^2/\Gs)$ & geod. periods \\
  \hline
	\hline
	 &\multicolumn{2}{|l|}{($G<\SOr(3)$, $G=\Gp$)} &$(\cong S^2/G)$ &$(\cong S^2/\Gs)$&\\
	\hline
   &$(\mathrm{C}_{p})$& $S^2(p,q)$ & & &    \\ 
 1. && $2|(p+q)$, $2|pq$ & $S^2/\mathrm{C}_{(p+q)/2}$& $S^2/\mathrm{C}_{(p+q)/2\mathrm{h}}$ & $(1_{\infty},\overline{(p+q)/2})$   \\
  && $2|(p+q)$, $2\not|pq$  & $S^2/\mathrm{C}_{(p+q)/2}$& $S^2/\mathrm{S}_{p+q}$ & $(\overline{(p+q)/2})$   \\
 1'. && $2\not|(p+q)$, $2|pq$ & $S^2/\mathrm{C}_{p+q}$ &$S^2/\mathrm{C}_{p+q\mathrm{h}}$& $(1_{\infty},\overline{p+q})$   \\
 2. &$\mathrm{D}_{2n}$& $S^2(2,2,2n)$ & $S^2/\mathrm{D}_{2n}$ &$S^2/\mathrm{D}_{2n\mathrm{h}}$& $(1_{\infty},2,2,2n)$  \\
 3. &$\mathrm{D}_{2n+1}$& $S^2(2,2,2n+1)$&$S^2/\mathrm{D}_{2n+1}$&$S^2/\mathrm{D}_{2n+1\mathrm{d}}$   &$(1_{\infty},\overline{2},2n+1)$ \\
 4. &$\mathrm{T}$& $S^2(2,3,3)$ & $S^2/\mathrm{T}$ &$S^2/\mathrm{T_h}$&$(1_{\infty},2,\overline{3})$  \\
 5. &$\mathrm{O}$& $S^2(2,3,4)$ & $S^2/\mathrm{O}$ &$S^2/\mathrm{O_h}$&$(1_{\infty},2,3,4)$ \\
 6. &$\mathrm{I}$& $S^2(2,3,5)$ & $S^2/\mathrm{I}$ &$S^2/\mathrm{I_h}$&$(1_{\infty},2,3,5)$  \\
\hline
	 &\multicolumn{2}{|l|}{($-1 \notin G \not < \SOr(3)$, $G\cong \Gt$)} &$(\cong S^2/\Gt)$ &$(\cong S^2/\Gs)$&\\
	\hline
  & $(\mathrm{C}_{p\mathrm{v}})$& $D^2(;p,q)$ &&&  \\
 7. &&   $2 |(p+q)$, $2|pq$ &$S^2/\mathrm{D}_{(p+q)/2}$&$S^2/\mathrm{D}_{(p+q)/2\mathrm{h}}$&$(1_{\infty},(p+q)/2)$  \\
  &  &$2 |(p+q)$, $2\not|pq$ &$S^2/\mathrm{D}_{(p+q)/2}$&$S^2/\mathrm{D}_{(p+q)/2\mathrm{d}}$&$(1_{\infty},(p+q)/2)$  \\
 7'. && $2\not|(p+q)$, $2|pq$ & $S^2/\mathrm{D}_{p+q}$ &$S^2/\mathrm{D}_{p+q\mathrm{h}}$& $(1_{\infty},2,p+q)$   \\
 8. &$\mathrm{S}_{4n}$ &$\R \mathbb{P}^2(2n)$ &$S^2/\mathrm{C}_{4n}$&$S^2/\mathrm{C}_{4n\mathrm{h}}$& $(1_{\infty},\overline{4n})$   \\
 9. &$\mathrm{C}_{2n+1\mathrm{h}}$ &$D^2(2n+1;)$ &$S^2/\mathrm{C}_{4n+2}$ &$S^2/\mathrm{C}_{4n+2\mathrm{h}}$&  $(1_{\infty},\overline{2n+1})$ \\
 10. &$\mathrm{D}_{2n\mathrm{d}}$ &$D^2(2;2n)$ & $S^2/\mathrm{D}_{4n}$ &$S^2/\mathrm{D}_{4n\mathrm{h}}$&$(1_{\infty},2,4n)$ \\
 11. & $\mathrm{D}_{2n+1\mathrm{h}}$&$D^2(;2,2,2n+1)$ &$S^2/\mathrm{D}_{4n+2}$  &$S^2/\mathrm{D}_{4n+2\mathrm{h}}$& $(1_{\infty},2,2n+1)$\\
 12. &$\mathrm{T_d}$ &$D^2(;2,3,3)$ & $S^2/\mathrm{O}$ &$S^2/\mathrm{O}_{\mathrm{h}}$& $(1_{\infty},3,4)$ \\
\hline
	 &\multicolumn{2}{|l|}{($-1\in G= \Gs\cong \Gp \times \Z_2$)} &$(\cong S^2/\Gp)$ &$(\cong S^2/G)$&\\
	\hline
13.& $\mathrm{S}_{\mathrm{4n+2}}$&$\R \mathbb{P}^2(2n+1)$& $S^2/\mathrm{C}_{2n+1}$&$S^2/\mathrm{S}_{\mathrm{4n+2}}$&$(\overline{2n+1})$ \\
14.& $\mathrm{C}_{2n\mathrm{h}}$ &$D^2(2n;)$ &$S^2/\mathrm{C}_{2n}$ &$S^2/\mathrm{C}_{2n\mathrm{h}}$&  $(1_{\infty},\overline{n})$ \\
15.&$\mathrm{D}_{2n+1\mathrm{d}}$ &$D^2(2;2n+1)$ & $S^2/\mathrm{D}_{2n+1}$  &$S^2/\mathrm{D}_{2n+1\mathrm{d}}$& $(1_{\infty},2n+1)$ \\
16.&$\mathrm{D}_{2n\mathrm{h}}$ &$D^2(;2,2,2n)$ &$S^2/\mathrm{D}_{2n}$ &$S^2/\mathrm{D}_{2n\mathrm{h}}$& $(1_{\infty},n)$  \\
17. &$\mathrm{T_h}$ &$D^2(3;2)$ & $S^2/\mathrm{T}$ &$S^2/\mathrm{T_h}$&$(1_{\infty},\overline{3})$\\
18.&$\mathrm{O_h}$ &$D^2(;2,3,4)$ & $S^2/\mathrm{O}$ &$S^2/\mathrm{O_h}$& $(1_{\infty},2,3)$\\
19.&$\mathrm{I_h}$ &$D^2(;2,3,5)$& $S^2/\mathrm{I}$ &$S^2/\mathrm{I_h}$&$(1_{\infty},3,5)$ \\

  \hline 
 \end{tabular}
\caption{Orbifolds of geodesics and (labeled) geodesic periods of Besse $2$-orbifolds. For definitions and notations see Sections \ref{sub:Riem_orb} and \ref{sub:orb_of_geo}. Expressions that only hold in the good orbifold case, i.e. when $p=q$, are stated in parentheses. For the good orbifolds $\Orb=S^2/G$, $G<\mathrm{O}(3)$, appearing in the table the second column specifies $G$ in terms of the Sch\"{o}nflies notation, see e.g. \cite{LaLi}). Recall that $\Gp = G\cap \SOr(3)$, $\Gt = \{g \in G| \mathrm{det}(g)g\}$ and $ G^*=\left\langle G,-1 \right\rangle$. A detailed discussion of the finite subgroups of $\Or(3)$ and their relations can for instance be found in \cite{LaLi}.}
\label{tab:gls} 
\end{table}

%

\end{document}